\documentclass[11pt,reqno]{amsart}

\usepackage{amsmath,amssymb,amscd,amsthm,verbatim,alltt,amsfonts,array}

\usepackage{caption} 
\usepackage{subfigure}  
\usepackage{color}
\usepackage{anyfontsize} 

\usepackage[numbers,square]{natbib}
\usepackage[backref=page]{hyperref} 

\usepackage{tikz}
\usepackage{tikz-cd}
\usetikzlibrary{arrows, matrix}

\advance \topmargin by-\headheight
\advance \topmargin by-\headsep
\evensidemargin \oddsidemargin
\theoremstyle{plain}

\usepackage{multirow}
\usepackage{booktabs}
\fboxsep=-\fboxrule
\usepackage{scalerel}

\def\Legenda014{\scalerel*{\includegraphics{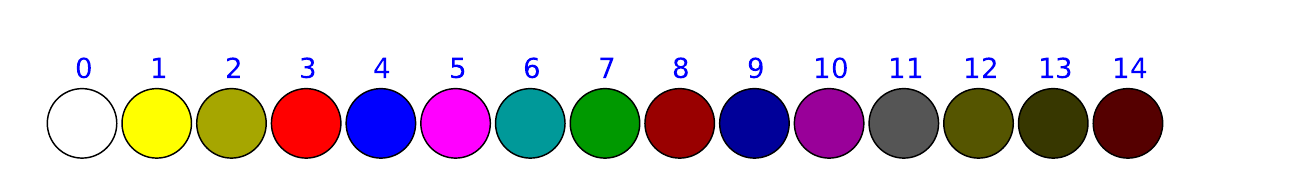}}{X\rule[-0.81ex]{0pt}{22pt}}}
\def\Legend_0-10{\scalerel*{\includegraphics{Legend_0-10.pdf}}{X\rule[-.6ex]{0pt}{1pt}}}
\def\LEGEND_0-7{\scalerel*{\includegraphics{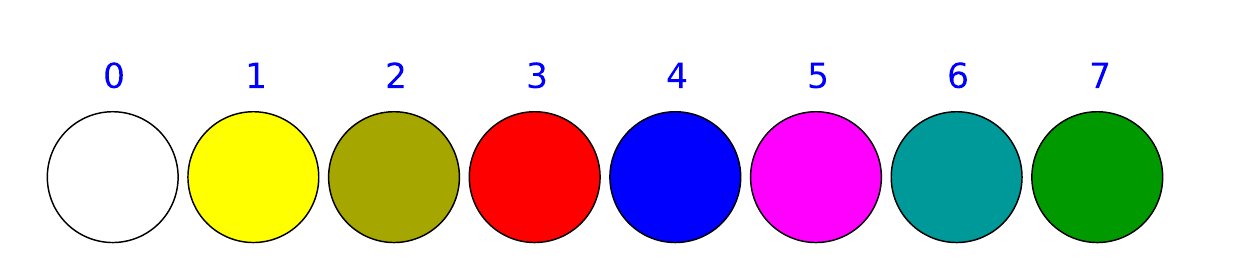}}{X\rule[-1.2ex]{0pt}{26pt}}}
\def\Legend_0-3{\scalerel*{\includegraphics{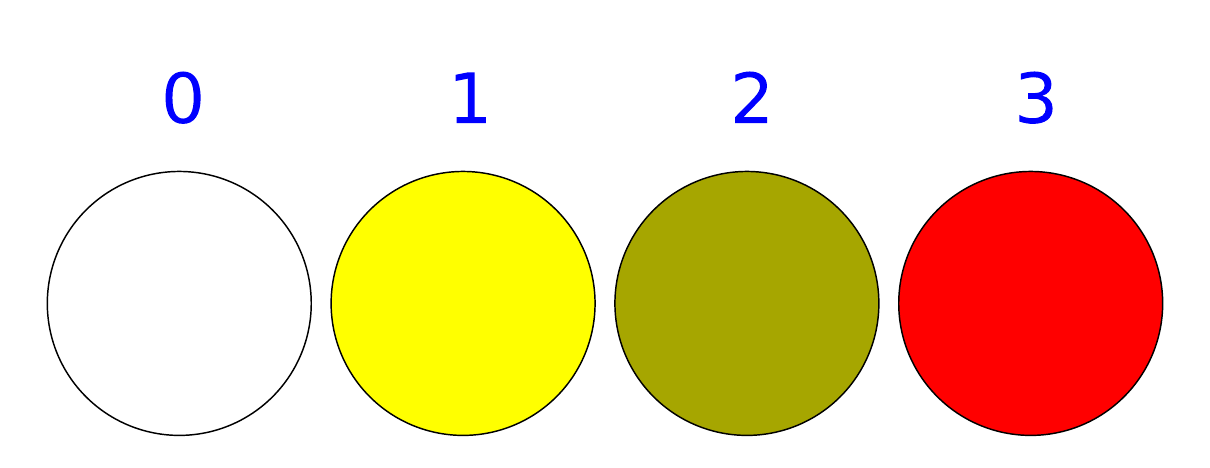}}{X\rule[-1.2ex]{0pt}{20pt}}}
\def\DiagramaComutativa{\scalerel*{\includegraphics{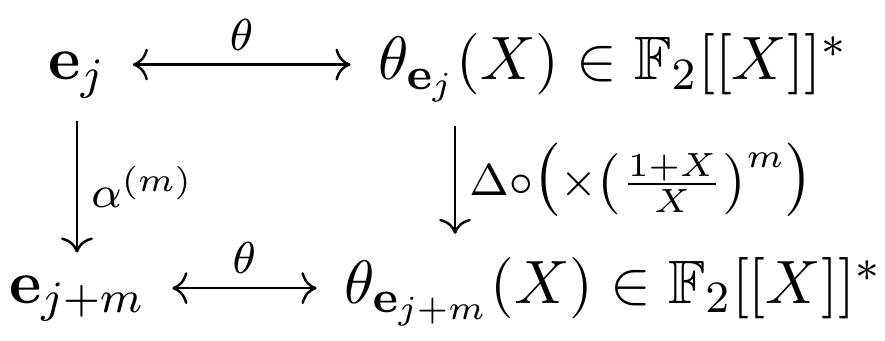}}{X\rule[
11ex ] { -333pt } {
20pt } } }

\newcommand{\cA}{\mathcal A}

\newcommand{\cS}{\mathcal S}
\newcommand{\cV}{\mathcal V}

\newcommand{\fS}{\mathfrak S}

\newcommand{\bu}{\mathbf u}

\newcommand{\bw}{\mathbf w}
\newcommand{\be}{\mathbf e}

\newcommand{\fp}{\mathfrak{p}}
\newcommand{\fq}{\mathfrak{q}}

\newcommand{\FF}{\mathbb F}
\newcommand{\NN}{\mathbb N}
\newcommand{\NNs}{{\mathbb{N}^*}}

\newcommand{\PP}{\mathbb P}

\newcommand{\ZZ}{\mathbb Z}

\newcommand{\Nwords}{\mathbb{N}^{+}}
\newcommand{\concat}{\mathbin{+\mkern-10mu+}}
\newcommand{\TNNsdoi}{T_{\NNs\mkern-3mu, \mkern1mu 2}}
\newcommand{\mycdots}{\!\!\!\!\cdot\!\cdot\!\cdot\!\!\!\!}

\newcommand{\Gin}{{\mathcal{G}_{in}}}

\newcommand\rs[1]{\mbox{{\color{red}\footnotesize #1}}} 
\newcommand\w[1]{\makebox[1.05em]{$#1$}} 

\newtheorem{theorem}{Theorem}
\newtheorem{lemma}{Lemma}
\newtheorem{corollary}{Corollary}

\newtheorem{conjecture}{Conjecture}

\newtheorem{proposition}{Proposition}

\newtheorem{paradoxproblem}{Paradox Problem}

\theoremstyle{remark}

\theoremstyle{plain}

\newcommand{\ind}{\operatorname{ind}}

\numberwithin{equation}{section}

\title{A growth model based on the arithmetic $Z$-game}

\author{Cristian Cobeli, Mihai Prunescu, Alexandru Zaharescu}

\address{Cristian Cobeli, 
\textit{Simion Stoilow}, Institute of Mathematics of the Romanian Academy,
21 Calea Grivi\c tei Street, 010702;
P.O. Box 1-764, 014700, Bucharest, Romania}
\email{cristian.cobeli@imar.ro}

\address{Mihai Prunescu, 
\textit{Simion Stoilow}, Institute of Mathematics of the Romanian Academy,
21 Calea Grivi\c tei Street, 010702;
P.O. Box 1-764, 014700, Bucharest, Romania}
\email{mihai.prunescu@imar.ro}

\address{Alexandru Zaharescu,
Department of mathematics, University of Illinois, 1409 West Green Street, Urbana ,
IL 61801, USA, and 
\textit{Simion Stoilow}, Institute of Mathematics of the Romanian Academy,
21 Calea Grivi\c tei Street, 010702;
P.O. Box 1-764, 014700, Bucharest, Romania}

\email{zaharesc@illinois.edu}



\begin{document}

\begin{abstract}
We present an evolutionary self-governing model based on the numerical atomic rule
$Z(a,b)=ab/\gcd(a,b)^2$, for $a,b$ positive integers. 
Starting with a sequence of numbers, the initial generation $\Gin$, a new sequence 
is obtained by applying the $Z$-rule to any neighbor terms. Likewise, applying repeatedly the same 
procedure to the newest generation, an entire matrix $T_\Gin$ is 
generated. 
Most often, this matrix, which is the recorder of the whole process, shows a fractal aspect and has
intriguing properties.

If $\Gin$ is the sequence of positive integers, in the associated matrix remarkable are the 
distinguished geometrical figures called the 
$Z$-solitons and the sinuous evolution of the size of numbers on the western edge. 
We observe that $T_\NNs$
is close to the analogue free of solitons matrix generated from an initial generation in which 
each natural number is replaced by its largest divisor that is a product of distinct primes. We 
describe the shape and the properties of this new matrix.

N. J. A. Sloane raised a few interesting problems regarding the western edge of the matrix 
$T_\NNs$.
We solve one of them and present arguments for a precise conjecture on another.

\end{abstract}

\maketitle

\section{Introduction Story}
Many different mathematical models have been proposed to study an evolutionary self-governing 
system. 
In the last several decades, a particular attention was devoted to those that are based on simple 
generating rules that produce complex outcomes. Such an example is the growing model based on the 
numerical $Z$-rule introduced 
in~\cite{CZ2013}
\begin{equation}\label{eqZrule}
	\begin{split}
	Z(a,b)=\frac{ab}{\gcd(a,b)^2}\,,\quad\text{ $a,b\in\NNs$,}
	\end{split}
\end{equation}
where $\NNs:=\NN\setminus\{0\}$.
The numbers are recorded in cells and, for simplicity, we keep the unidirectional development of 
future generations, composed of children 
$Z(a,b)$ born from parents $a$ and $b$, which are neighbor cells in the previous generation. 

For a plastic representation of the $Z$-rule~\eqref{eqZrule}, one can think that any cell 
containing a positive integer $n$ is a citadel composed of towers. There are as many towers in 
the citadel as prime factors $n$ has.
Each tower is associated to a prime and the height of the tower corresponding to a prime 
$p$ that divides  $n$  equals the power of $p$ in the factor decomposition of $n$. 
In particular, the citadel of a cell containing the number $1$ has no towers at all. 
Likewise, one may think that the citadel $n$ has towers associated to the primes that do not divide 
$n$ also, but these towers have zero height.
Then, the $Z$-rule topples the towers of the neighbor citadels $a$ and $b$ creating a new citadel 
$Z(a,b)$ in the next generation. The towers of the 
new citadel have heights equal with the absolute difference of the heights of towers corresponding 
to the same prime in $a$ and $b$ and, if a prime divides only one of $a$ and $b$, then this tower 
is reproduced unchanged in the new citadel.

The process starts with a sequence of numbers $\Gin$, which may be finite or not, which are placed 
in a row of cells.
This sequence is called the initial generation and the $Z$-rule is applied sequentially on 
each two consecutive terms of  $\Gin$. Whence, a new generation is born and its cells are placed in 
the following row. 
Usually, in graphic representations, we slightly shift to the right the new generation such that 
any new cell is placed in the middle under its parents.
Repeating the process, we obtain a matrix $T_{\Gin}$ with infinitely 
many rows if $\Gin$ is infinite. As the reproducing rule remains unchanged, the results depend 
only on the initial 
generations and we shall see that in this way a large variety of outcomes are produced.

The matrices of numbers $T_{\Gin}$ have lots of features of which some are 
similar to the objects created by the abelian sandpile model proposed by 
Bak, Tang and Wiesenfeld~\cite{BTW1987}. The intensely studied model, also called the 
chip-firing game, was surveyed by Levine and Propp~\cite{LPro2010}.
Our $Z$-model also captures features of other evolutionary systems such as the Ducci-type 
game~\cite{CM1937}, \cite{CCZ2000}, \cite{CGVZ2002},  \cite{CT2004}, \cite{GVZ2005},  
\cite{CZZ2011}, \cite{HKSW2014}, \cite{BGS2015},
 the numerical 
ensembles created by median insertions, such as those related to  Pascal triangle~ \cite{Gra1992}
\cite{Gra1997}, \cite{Pru2011a}, \cite{Pru2011b}, \cite{Pru2012}, \cite{Pru2013}, \cite{CZ2013}
or the Farey sequences~\cite{CZ2003}, \cite{CZ2006}.
For many initial generations $\Gin$, the matrices $T_{\Gin}$ show complex self-similar structures,
like those of some particular abelian sandpile states 
\cite{LKG1990},\cite{BR2002}, \cite{CPS2008},\cite{SD2010}
or the outcomes produced in the related rotor-router model~\cite{PDK1996}, \cite{LP2008}, 
\cite{Pro2010}. 

A special feature of matrices $T_\Gin$ is the fact that they can be localized. For any prime $p$, 
the \textit{$p$-tomography} is the matrix of citadels of $T_\Gin$ in which all towers, 
except the towers associated to $p$, are deleted. Then $T_{\Gin}$ is the superposition (the 
element-wise multiplication) of the $p$-tomographies for all primes $p$, since the evolution 
according to the $Z$-rule is independent to one another.

We have already proved in~\cite{CZ2014} that the $Z$-rule produces objects with a fractal aspect if 
$\Gin$ is the sequence of prime numbers or its localized slice, the
sequence of zeros except one term that is equal to a prime $p$.
In this article we show that this also happens if the initial generation is the $p$-spaced sequence 
$\cA_p=\{\fp_n\}_{n\ge 1}$, where 
\begin{equation}\label{eqAP1}
	\begin{split}
     \fp_n=\begin{cases}
              p \quad&\text{if $p\mid n$,}\\
              0 \quad&\text{else,}
         \end{cases}	
	\end{split}
\end{equation}
or $\cV_p=\{\fq_n\}_{n\ge 1}$, the $p$-section of positive integers,
\begin{equation}\label{eqAP2}
	\begin{split}
     \fq_n=\begin{cases}
              p^{v_p(n)} \quad&\text{if $p\mid n$,}\\
              0 \quad&\text{else,}
         \end{cases}	
	\end{split}
\end{equation}
where $v_p(n)$ is the $p$\textit{-valuation} of $n$, which
is, the power of $p$ into the prime decomposition \mbox{of $n$}. 
The classic Sierpinski fractal appears if $p=2$ (a finite fragment is shown in 
Figure~\ref{FigTomo2}), while more complex self-similar 
patterns are typical for larger primes (see Figures~\ref{FigTomo13-19}--\ref{FigTomo5}).
Always, in a graphical representation, we present only a triangular region of $T_\Gin$, the one 
composed by the cells 
born in future generations from the part of $\Gin$ shown on the first row.

%
\begin{figure}[t]
\small\footnotesize
 \centering
 \vspace*{-1.4em} 
\begin{equation*}
\arraycolsep=2.4pt\def\arraystretch{1.12}
\begin{array}{ccccccccccccccccccccccccc}
\w{} & \w{}  &\w{}& \w{} & \w{}& \w{} & \w{}& \w{} &\w{}& \w{} &\w{}& \w{} & \w{}& \w{} & \w{}& 
\w{} 
& \w{}& \w{} & \w{} &\w{} & \w{}& \w{}
& \w{}\\
\textbf{1} & & 2 & & 3 & & 4 & & 5 & & 6&  & 7 & & 8 & & 9 & & 10 & & 11 & & 12 & & \cdots\\
& \textbf{2} & & 6 & & 12 & & 20 & & 30 & & 42 & & 56 & & 72 & & 90 & & 110 & & 132  & & \cdots& \\
& & \textbf{3} & & 2 & & 15& & 6 & & 35 & & 12 & & 63 & & 20 & & 99 & & 30 & & \cdots & & \\
& & & \textbf{6} & & 30& & 10 & & 210 & & 420 & & 84& & 1260 & & 1980 & & 330 & & \cdots &  & & \\
& & & & \textbf{5} & & 3 & & 21 & & 2 & & 5 & & 15 & & 77 & & 6 & & \cdots & & & & \\
& & & & & \textbf{15} & & 7 & & 42 & & 10& & 3 & & 1155 & & 462 & & \cdots & &  & & & \\
& & & & & & \textbf{105} & & 6 & & 105 & & 30 & & 385 & & 10 & & \cdots & &  & & & & \\
& & & & & & & \textbf{70} & & 70 & & 14 & & 462 & & 154 & & \cdots &  &  & & & & & \\
& & & & & & & & \textbf{1}& & 5 & & 33 & & 3 & & \cdots & & &  & & & & & \\
& & & & & & & & & \textbf{5}& & 165 & & 11 & & \cdots & & & & & & & & & \\
& & & & & & & & & & \textbf{33} & & 15 & & \cdots & & & & & & & & & & \\
& & & & & & & & & & & \textbf{55} & & \cdots & & & & & & & & & & & \\
& & & & & & & & & & &  & {\boldsymbol\cdots} &  & & & & & & & & & & & \\     
\end{array}
\end{equation*}
 \caption{The matrix $T_{\NNs}$. After the first generation, the rows are shifted to the 
right so that any child is placed in the middle, under its parents.}
 \label{FigureTNNs}
 \end{figure}

A fundamental problem raised by the $Z$-model concerns the shape and properties of the matrix grown 
from the first generation $\NNs$. Its north-west corner is shown in Figure~\ref{FigureTNNs}. The 
$p$-tomographies of $T_\NNs$ can be grown individually, by starting in 
the first row with the sequence $\cV_p$. Of particular interest are questions related to the 
geometrical figures propagated from the cell $p^g,\ g\ge 2$.  
For any prime $p$ and integer $g\ge 2$, we denote by $\fS(p,g)$ the collection of connected cells 
that starts from $p^g$ and contains only powers of $p$ larger than two. 
We call these figures \textit{$Z$-solitons} and present two of them in Figure~\ref{FigSolitons}.
Unlike the set of cells containing powers of $p$  less than two, which forms a continuous texture 
all over the infinite matrix $T_\NNs$, the solitons are larger and larger with the power $g$, but 
finite.

\begin{figure}[ht]
 \centering
 \mbox{
 \subfigure{
    \includegraphics[width=0.48\textwidth]{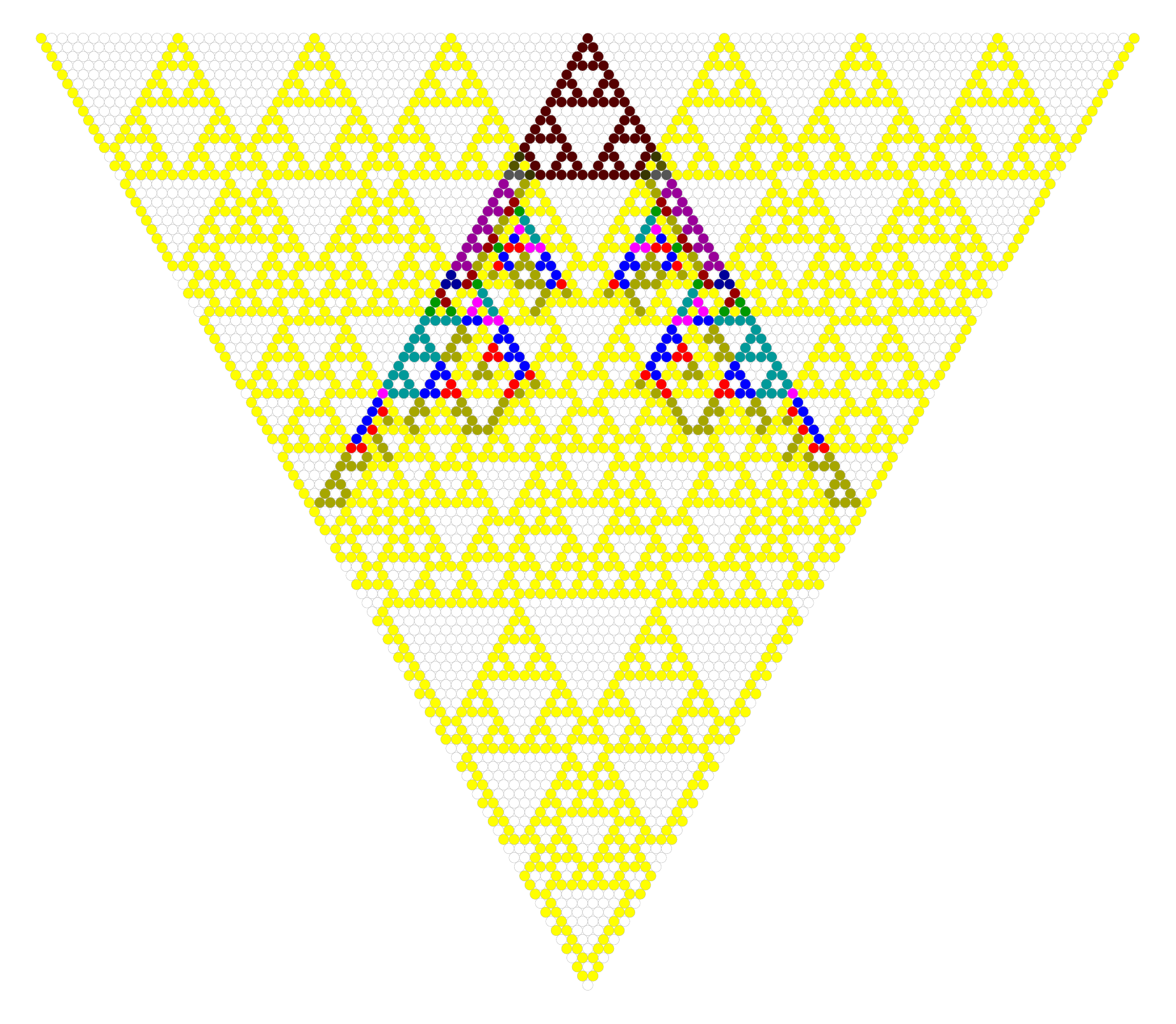}
 }
 \subfigure{
    \includegraphics[width=0.48\textwidth]{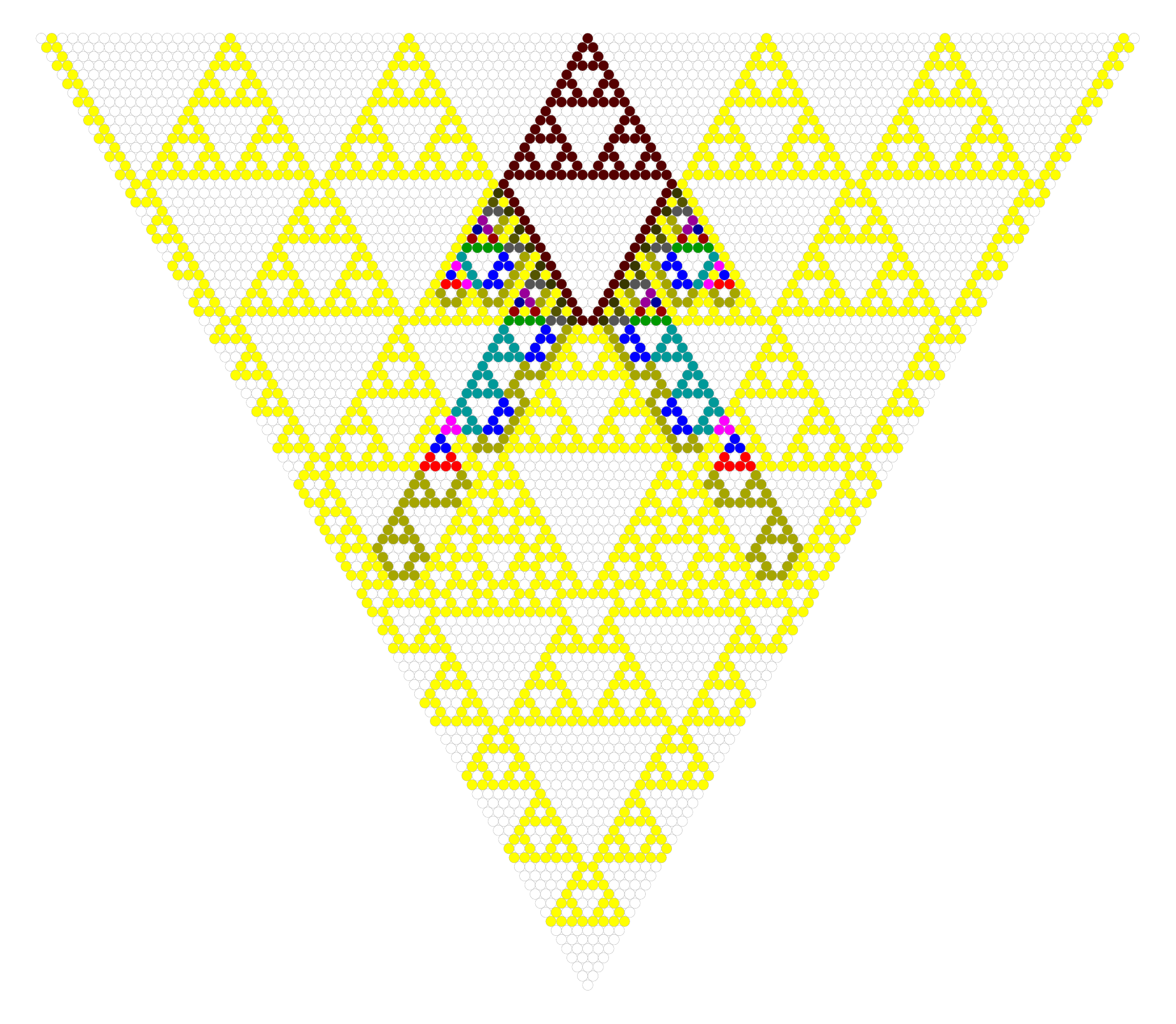}
 }
 }
\caption{The solitons $\fS(p,n)$ generated by the cells $p^{14}$ with $p=13$ and $p=17$.
 	  \vspace*{-1.11em}  \\
     	  \mbox{\small Code of colors for the cells containing the powers of $p$ from 
$0$ to $14$:\protect \Legenda014}			
}
 \label{FigSolitons}
 \end{figure}

\begin{conjecture}\label{ConjectureSolitons}
     Two distinct solitons $\fS(p,g_1)$, $\fS(p,g_2)$, neither overlap nor touch each other.
\end{conjecture}

The characterization problem of the evolution of shape and size of the solitons for different 
$p$ and growing $g$  is the analog of the limiting shape problems 
studied in \cite{GQ2000}, \cite{FR2008}, \cite{LP2008}, \cite{LPer2010}, \cite{FR2011}.
The solitons $\fS(2,g)$ verify Conjecture~\ref{ConjectureSolitons}, as follows from their complete 
description presented in Section~\ref{SectionSoliton2}. 
For odd primes, 
$\fS(p,g)$ are more complex, with an aspect of certain fringed parts of Sierpinski triangles.
For a fixed prime $p$, the series of solitons $\fS(p,g)$, as $g$ increases, offers an intriguing 
spectacle and their growth appears to be proportional, as in the analog case of the abelian 
sandpiles~\cite{Ost2003}, \cite{DS2012}, \cite{DS2013}, \cite{DD2014}. 

\medskip
Another zone of interest is the first column or the \textit{west edge} of matrix $T_\Gin$. 
We denote this sequence by $W_\Gin$. The western 
edge can be viewed as the projection of $\Gin$ through the entire $Z$-process. 
Notice that the value of the $m$-th citadel on $W_\Gin$ is influenced by the values of all first 
$m$ citadels of $\Gin$ and by neither of the others. An example is the western frontier of 
the triangle in Figure~\ref{FigureTNNs}:
	\begin{equation*}
	\begin{split}
	W_\NNs :\ 1,2,3,6,5,15,105,70,1,5,33,55,65,273, 1001, \dots
	\end{split}
\end{equation*}
Various evidence obtained by computer verifications suggest that no square of any prime divides a 
term of this sequence. 
In other words, no soliton grows as large as to touch the western edge $W_\NNs$.
Corollary~\ref{Corolarry2} shows that this is true for solitons $\fS(2,g)$, $g\ge 2$.

At the exponents level, this is the counterpart of the Gilbreath's 
Conjecture~\cite[A10]{Guy2004}, \cite{Odl1993}, which refers to the similar process that starts 
with the sequence of primes as $\Gin$ and grows the future generations with children born by taking 
the absolute difference of their parents. The Gilbreath's Conjecture says that the west edge of the 
triangle composed of these rows of successive gaps of gaps, contains only ones.

\renewcommand{\thefootnote}{$\dagger$} 
Another example extending this widespread property~\cite{Odl1993} is the matrix that starts 
with the sequence of Sophie Germain primes\footnote{A positive integer $p$ is a Sophie
Germain prime if $p$ and $2p+1$ are primes at the same time.} and is generated in the same way, 
listing successively the gaps from the previous row of gaps.  For this matrix, John W. 
Layman~\cite[{\tt A080209}]{OEIS} observed and conjectured that the left edge consists only of $1$s 
\mbox{and $3$s.}

In our multiplicative setting, we conjecture that the maximal power of any prime that appears in 
the decomposition of the numbers situated on the left edge of $T_{\NNs}$ is one.
\begin{conjecture}[Section 9 \cite{CZ2013}]
   The left edge of the infinite triangle generated by the iterated application of the
$Z$-rule to the set of positive integers contain only square free numbers.
\end{conjecture}


The object of Sections \ref{SectionPP}--\ref{SectionWestPP} is to compare and analyze the 
similarities between $T_\NNs$ and an analogue matrix that has no  solitons.
The $p$-tomographies of this new matrix are generated by sequence 
$\cA_p$ 
 and in Theorem~\ref{TheoremPeriodicity} we show  that these tomographies are eventually periodic 
for all $p$. Furthermore,
we observe the closeness of the citadels on the western side of the two matrices and in 
Theorem~\ref{TheoremWPP2lap} we characterize 
the structure of the sinuous series of extreme values of the western edge of the matrix with no 
solitons.

Theorem~\ref{TheoremTNNsdoi} gives a complete characterization of  the 
$2$-tomography of  matrix $T_\NNs$. In particular, it shows that there are no fours on the 
western edge of $T_\NNs$, solving a problem raised by  N. J. A. Sloane~\cite[A222313]{OEIS}.
Our analysis in Sections \ref{SectionPP}--\ref{SectionWestPP} allows us to formulate the precise 
Conjecture~\ref{Conjecture3} regarding another problem raised by Sloane\cite[A222313]{OEIS}, 
\cite[Question 3]{CZ2014}.

\section{Notations}

Starting with a sequence of integers $\cS=\{s_1,s_2,\dots\}$, we consider the matrix whose first
row is $\cS$ and the following ones are generated with the $Z$-rule. We denote this matrix by
$T_{\cS}=(t_{j,k})_{1\le j,k}$, where $t_{1,1}=s_1$, $t_{1,2}=s_2,\dots$ and
\begin{equation*}
   t_{j,k}=Z(t_{j-1,k},t_{j-1,k+1}),\ \text{ for} \  2\le j,\ 1\le k.
\end{equation*}
If the initial sequence is a finite ordered set  $\cS=\{s_1,\dots, s_K\}$, we obtain the numerical 
triangle
\begin{equation*}
   T_{\cS}(K)=\big\{t_{j,k}\ \colon \  1\le j \le K,\ 1\le k\le K-j+1\big\},
\end{equation*}
whose first row is $t_{1,1}=s_1$, $t_{1,2}=s_2,\dots, t_{1,K}=s_K$, and following ones are
generated recursively by
\begin{equation*}
	  t_{j,k}=Z(t_{j-1,k},t_{j-1,k+1}), \quad\text{for $2\le j\le K$ and $1\le k\le K-j+1$.}
\end{equation*}

We say that $t_{j,k}=Z(t_{j-1,k},t_{j-1,k+1})$, for $j\ge 2$,  is the \textit{child} of its
\textit{parents}
$t_{j-1,k}$ and $t_{j-1,k+1}$ and in pictures we usually place the child in the middle, below its
parents.

The $j$-th row of the matrix is called the $j$-th \textit{generation} and we denote it by
\begin{equation*}
   G_{\cS}(j)=\big\{t_{j,k}\ \colon \  1\le k \big\}\quad\text{and}\quad
   G_{\cS}(j; K)=\big\{t_{j,k}\ \colon \  1\le k \le K\big\}\,,
\quad\text{for $j\ge 1$}\,.
\end{equation*}
We denote the \textit{west-side} of the triangle by
\begin{equation*}
   W_{\cS}=\big\{t_{j,1}\ \colon \  1\le j \big\}\quad\text{and}\quad
   W_{\cS}(K)=\big\{t_{j,1}\ \colon \  1\le j \le K\big\}\,.
\end{equation*}

The evolution at the exponents
level is presented into the following tables:
\begin{equation*}
     \begin{split}
             v_p(T_{\cS})&=\big\{v_p(t_{j,k})\ \colon \  1\le j, k \big\},\\
	     v_p(T_{\cS}(K))&=\big\{v_p(t_{j,k})\ \colon \  1\le j \le K,\ 1\le k\le K-j+1\big\}.
     \end{split}
\end{equation*}

Given  an infinite matrix  $T_{\cS}$ or a bounded triangle $T_{\cS}(K)$, we denote their 
$p$-\textit{tomography} 
(also called the $p$\textit{-slice} or the $p$\textit{-section}) by 
\begin{equation*}
     \begin{split}
             T_{\cS,p}&=\big\{p^{v_p(t_{j,k})}\ \colon \  1\le j, k \big\},\\
	     T_{\cS,p}(K)&=\big\{p^{v_p(t_{j,k})}\ \colon \  1\le j \le K,\ 1\le k\le K-j+1\big\}.
     \end{split}
\end{equation*}
Thus, the superposition of all $p$-slices recovers the full matrix:
\begin{equation*}
     \begin{split}
             T_{\cS}=\prod_{p}T_{\cS,p}\quad \text{and}\quad
	     T_{\cS}(K)=\prod_{p} T_{\cS,p}(K),
     \end{split}
\end{equation*}
where the product over all primes $p$ is taken component-wise.

For any positive integer $n$, we denote by $p(n)$ the largest square free number that divides $n$, 
and by $\PP$ the sequence of these numbers:
\begin{equation}\label{eqDefPP}
	\begin{split}
	p(n)=\prod_{p | n}p,\qquad 
	\mathbb{P} = \{p(n) : n\in\mathbb{N}\}\,.
	\end{split}
\end{equation}

We denote by $\FF_2[[X]]$ the ring of meromorphic series of variable $X$ and coefficients in  the
field with two elements $\FF_2$ and by $\FF_2[[X]]^*\subset \FF_2[[X]]$ the collection of series
that are sums of monomials $X^k$ with $k\ge 1$, only.

As usual, the number of distinct prime factors of $n$ is denoted by $\omega(n)$ and the notation 
for the multiplicative order of $a$ modulo $p$ is  $\ind_p(a)$.

\section{The $2$-tomography of $T_\NNs$}\label{SectionSoliton2}

The real action on $\TNNsdoi$ is on the exponents level and, to understand its result, we need to
formalize it.
Let $\Nwords$ denote the collection of nonempty finite words over the 
infinite alphabet $\NN$.
We introduce the following sequence of words in $\Nwords$, defined recursively:
\begin{equation*}
  x_1=0,\quad x_n = x_{n-1}\concat (n-1) \concat x_{n-1},\ \ \text{for $n \geq 2$},
\end{equation*}
where ``$\concat$'' denotes the concatenation of integers. 
Note that $x_n$ is the concatenation of $2^n-1$ integers.
Since $x_n$ is an initial sub-word of 
$x_{n+1}$, for all $n \geq 1$, there exists a limit sequence $\bw_0 : \NNs \rightarrow \NN$, 
whose first $ 2^n - 1$ terms coincides with the letters of $x_n$, for $n \geq 1$.
We write: $\bw_0 = \varinjlim x_n$.
     
Similarly, starting with $1$ instead of $0$, we define the sequence of words
\begin{equation*}
  y_1=1,\quad y_n = y_{n-1}\concat n \concat y_{n-1},\ \ \text{for $n \geq 2$}
\end{equation*}
and obtain the limit sequence  $\bw_1 = \varinjlim y_n$, 
whose first $ 2^n - 1$ terms coincides with the letters of $y_n$, for $n \geq 1$.

The first terms of $\bw_0$ and $\bw_1$ are:
\begin{equation*}
	\begin{split}
	\bw_0:&\ \ 
		0, 1, 0, 2, 0, 1, 0, 3, 0, 1, 0, 2, 0, 1, 0, 4, \dots \\
	\bw_1:&\ \ 	
		1, 2, 1, 3,  1, 2, 1, 4, 1, 2, 1, 3,  1, 2, 1, 5, \dots
	\end{split}
\end{equation*}

For a given sequence $a : \NNs \rightarrow \NN$, we denote  by $\alpha(a)=\{\alpha_n\}_{n\ge 1}$
\textit{the sequence of absolute differences} between consecutive terms:
\begin{equation*}
	\alpha_n = | a_{n+1} - a_n |,\quad \text{for}\ n \geq 1
\end{equation*}
and by $\beta(a)=\{\beta_n\}_{n\ge 1}$ \textit{the bubbled 
sequence}, defined by
\begin{equation*}
	\beta_{2n-1} = \beta_{2n} = a_n,\quad \text{for}\ n \geq 1\,.
\end{equation*}
We use the same notations for the similar operations applied on words, where the action is on the 
the sequences of their letters. For example: 
$\alpha(x_2)=\alpha(010)=11$	and 
$\beta(y_2)=\beta(121)=112211$.

\medskip
\begin{figure}[bh]
 \centering
    \includegraphics[width=0.78\textwidth]{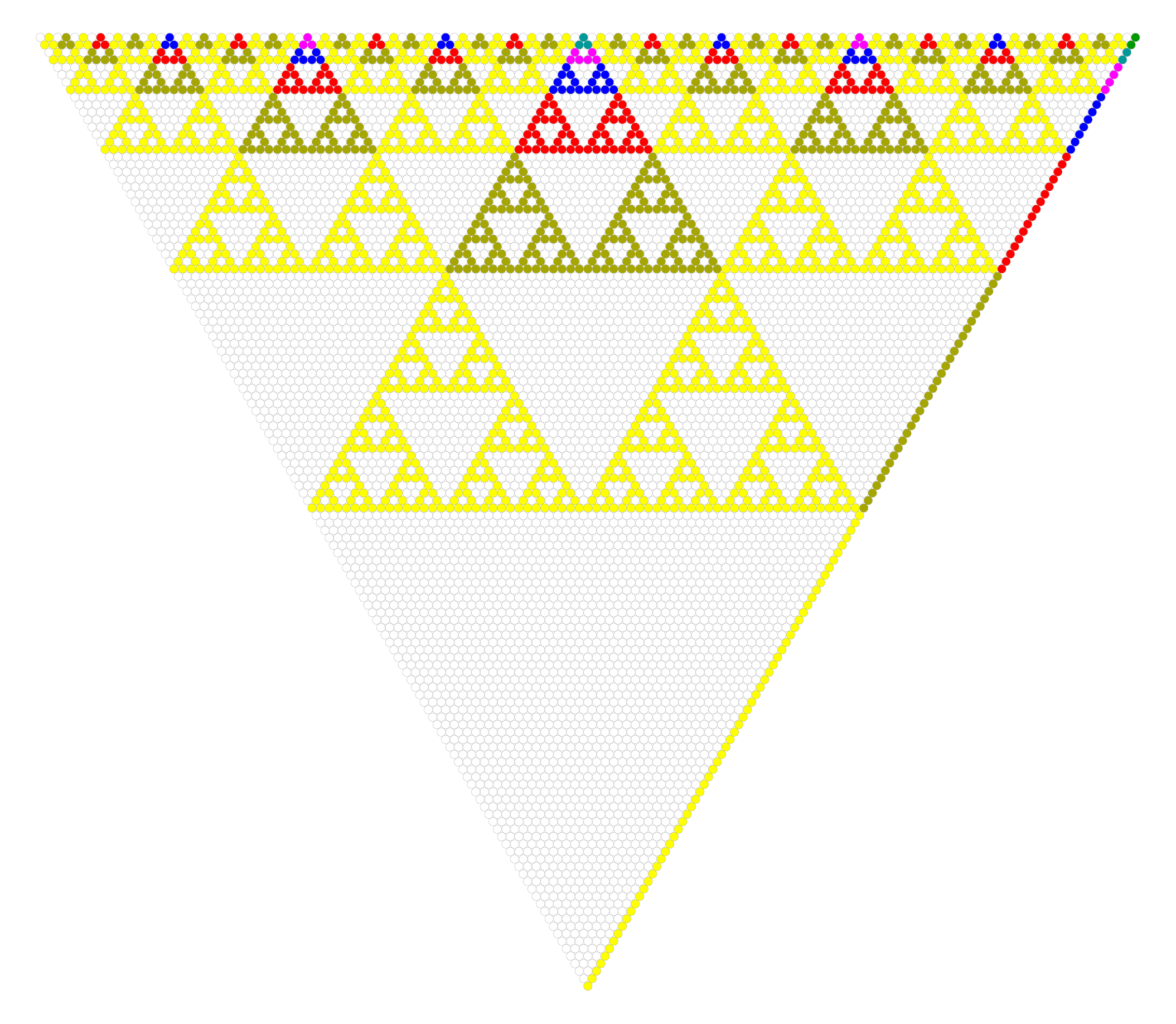}
 \caption{The tomography of $T_{\NNs}(129)$ for $p=2$. Notice that the larger solitons are further 
and further away from the western edge of $T_{\NNs}$.
	  \vspace*{-.95em}  \\
     	  \mbox{\small Code of colors for the cells containing the powers of $2$ from 
$0,\dots,7$:\protect \LEGEND_0-7}
}
 \label{FigTomo2}
 \end{figure}

The next lemma shows how $\bw_0$ and $\bw_1$ relates to one another through these operations.

\begin{lemma}\label{LemmaW01} The 
following properties hold true: 
\begin{enumerate}
\item $\bw_1 = \bw_0 + 1$, with element-wise addition;
\item $\bw_0 = v_2(\NNs)$, with term-wise application of the valuation $v_2$;
\item $\bw_1 = v_2(2\NNs)$;
\item $\alpha(\bw_0) = \alpha(\bw_1) = \beta(\bw_1)$.
\end{enumerate}
\end{lemma}

\begin{proof}
 (1) The equality  $\bw_1 = \bw_0 + 1$ follows directly from the definitions. 
 
 (2)
We prove the equality by induction. The initial step: $\bw_0(1) = 0 = v_2(1)$.
Now, suppose that  $\bw_0(n) = v_2(n)$ for all $n \in \{1, 2, \dots , 2^k-1\}$. 
Then $v_2(2^k) = k$ and $v_2(2^k + m) = v_2(m)$, 
for all  $m \in \{1, 2, \dots , 2^k-1\}$, by the definition of the valuation.
Using the definition of $\bw_0$, this means that   $\bw_0(n) = v_2(n)$, for
$n \in \{1, 2, \dots , 2^{k+1}-1\}$, as needed. 

 (3) $\bw_1 = v_2(2\NNs)$ follows from (1) and (2).

 (4) The equality
$\alpha(\bw_0) = \alpha(\bw_1)$ also follows directly from (1) and the definition of 
$\alpha(\cdot)$.
So it remains to prove  that $\alpha(\bw_1) = \beta(\bw_1)$. 
We proceed by induction. 

The initial step: applying $\alpha(\cdot)$ to the finite sequence $1,2,1$ (which are the letters 
of $y_2$, the beginning of $\bw_1$),  we get $1,1$, the beginning of $\beta(\bw_1)$ or, with the 
notation on words, $\alpha(y_2)=\alpha(121)=11=\beta(y_1)$.

The induction step: suppose that $\alpha(y_n)=\beta(y_{n-1})$. Then
\begin{equation*}
	\begin{split}
	\alpha(y_{n+1})= \alpha(y_n\concat (n+1)\concat y_n)
	= \beta(y_{n-1})\concat n\concat n\concat \beta(y_{n-1}) = \beta(y_n),
	\end{split}
\end{equation*}
since the last and the first letters of $y_{n-1}$ are equal to $1$. 
This completes the proof of the lemma.
\end{proof}

Thus, by Lemma \ref{LemmaW01} we see that the sequences of gaps between consecutive terms of 
$\bw_0$ and $\bw_1$ both coincide with the bubbled  
sequence
\begin{equation*}
	\begin{split}
	\beta(\bw_1): 
		1, 1, 2, 2,  1, 1, 3, 3, 1, 1, 
2, 2,  1, 1, 4, 4, 1, 1, 2, 2,  1, 1, 3, 3, 1, 1, 2, 2,  1, 1, 5, 5, \dots
	\end{split}
\end{equation*}

With the above notations, we see that the sequence of exponents of $2$ on the first row of 
$\TNNsdoi$ coincides with $\bw_0$. Then, by Lemma~\ref{LemmaW01}, it follows that the subsequent 
sequences of exponents of $2$ on the following rows of $\TNNsdoi$ are: 
	$\alpha(\bw_0), \alpha^{(2)}(\bw_0), \alpha^{(3)}(\bw_0),\dots$ In general, the $m$th row of 
$\TNNsdoi$ is
\begin{equation*}
	\begin{split}
	2^{\alpha^{(m-1)}(\bw_0)(1)},\ 2^{\alpha^{(m-1)}(\bw_0)(2)},\ 2^{\alpha^{(m-1)}(\bw_0)(3)},\ 
	\dots ,	\quad\text{for $m\ge 1$.}							
	\end{split}
\end{equation*}

Now we can describe the structure of the matrix of the exponents $v_2(T_\NNs)$, which corresponds
explicitly to the explicit description of the $2$-tomography of $T_\NNs$.
Its initial cut-off triangle
$\TNNsdoi(129)$, composed of $129$ rows, is shown in Figure~\ref{FigTomo2}. We see that,
geometrically, it is part of an infinite Sierpinski triangle. Notice that the horizontal rows
are grouped naturally in \textit{slices} containing sequences of pairs of triangles. The couple of
triangles in each pair is colored with the same color and the change of colors from a pair to
another corresponds to the change of numbers in the \mbox{sequence  $\beta(\bw_1)$.}

The sequence of slices $\{S_k\}_{k\ge 0}$ in which $v_2(T_\NNs)$ is partitioned are larger and
larger in size. 
The slice $S_0$ is just the first row and it is exceptional. 
The next slice, $S_1$, is the second row. Then, for any $k\ge 2$, the slice $S_k$ groups
$2^{k-1}$ rows, those from the $(2^{k-1}+1)$th  till the $2^k$th.

In any slice, the largest triangles formed by cells of the same color are the top rows of the
Pascal arithmetic triangle modulo $2$, with the odd entries replaced by a certain positive integer. 
Such a triangle depends on two parameters: the height $h$ and the weight $t$, which is the value
of the non-zero entries. We denote it by $P_2(h,t)$ (see the left triangle
in Figure~\ref{FigureP2ht} for such an example).

Triangle $P_2(h,t)$ is generated as Pascal's classic triangle, by starting from the top with
a symbolic variable $t$, which satisfies the rule $t+t=0$. The same result is obtained if the top
is placed somewhere in a row of zeros (see the matrix from the right-side of
Figure~\ref{FigureP2ht}).

\begin{figure}[ht]
\bigskip
\begin{minipage}[c]{0.3\linewidth}
\centering
\small
 \vspace*{-1.4em} 
\begin{equation*}
\arraycolsep=2.4pt\def\arraystretch{1.02}
\begin{array}{ccccccccccccc}
& &   &  &   &  & 10 &   &  &  &  &  &   \\ 
& &   &  &   & 10 &  & 10  &  &  &  &  &   \\ 
& &   &  &  10 &  &  0 &   & 10 &  &  &  &   \\ 
& &   & 10 &  & 10 &  & 10  &  & 10 &  &  &   \\ 
& &  10  & & 0 &  & 0 &  & 0 & & 10 &  &   \\  
& 10 &  &10 & & 0 &  & 0 &  & 10 &  & 10 &   \\   
10 & & 0 & & 10 &  & 0 &  & 10 & & 0 & & 10  \\     
\end{array}
\end{equation*}
\label{FigureP2710}
\end{minipage}
\hspace{11mm}
\begin{minipage}[c]{0.53\linewidth}
\centering
\small
 \vspace*{-1.4em} 
\begin{equation*}
\arraycolsep=2.4pt\def\arraystretch{1.42}
\begin{array}{ccccc ccccc c ccccc ccccc}
  \mycdots\ \, & & 0 & & 0 & & 0 & & 0 & & t & & 0 & & 0 & & 0 & & 0 & &\ \, \mycdots\\
  & \mycdots &  & 0 &  & 0 &  & 0 &  & t &  & t &  &  0 &  &0 &  & 0 &  & \mycdots & \\
  & &  \mycdots &  & 0 &  & 0 &  & t &  & 0 & &  t &   &  0 & &  0 &  & \mycdots &  & \\
  & &  &  \mycdots &  & 0 &  & t &  & t & & t &   &  t & &  0 &  &   \mycdots & & &\\
  & &  &   &  \mycdots & & t &  & 0 & & 0 &  &  0 &  & t &   &  \mycdots &   & & &\\
\end{array}
\end{equation*}
\label{FigureP25t}
\end{minipage}
\caption{
\texttt{Left:} The triangle $P_2(7,10)$ of height $7$ and weight $10$.\\ 
\texttt{Right:} 
A triangle $P_2(5,t)$ 
generated by a single non-zero cell of weight $t$ placed in the center of a string of zeros of
length at least $4+1+4=9$.}\label{FigureP2ht}
\end{figure}

\begin{lemma}\label{LemmaP2t} 
Let $h$ be a positive integer and let  $t$ be a formal variable. 
 Let $\bu=\{u_k\}_{k\ge 1}$ be a sequence of zeros, except one term $u_n = t$ and suppose that 
$n \geq h$. Then, the matrix with rows $\bu, \alpha(\bu),
\alpha^{(2)}(\bu)\dots, \alpha^{(h-1)}(\bu)$ contains triangle $P_2(h,t)$.
\end{lemma}
\begin{proof}
     It suffices to note that condition $h\le n$ ensures that the object that develops from 
$u_n = t$ is not influenced by external obstacles and the operation of taking the absolute value of 
the difference acts on $\{0,t\}$  exactly as the operation that grows a Pascal triangle with entries 
in $\FF_2$.

\end{proof}

We summarize the complete description of $v_2(T_\NNs)$ (respectively $\TNNsdoi$) in the next
theorem.

\begin{theorem}\label{TheoremTNNsdoi}
{\normalfont (1)} Slice $S_0$ of the matrix $v_2(T_\NNs)$ is  sequence $\bw_0$. 
The next rows of $v_2(T_\NNs)$ are grouped in slices $S_k$, such that, for any $k\ge 1$, slice
$S_k$ is formed by  rows from the $(2^{k-1}+1)$th  till the $2^k$th.
{\normalfont (2)} The single row of $S_1$ is  sequence $\alpha(\bw_0)=\beta(\bw_1)$.
{\normalfont (3)} For any $k\ge 1$, the collection of non-zero elements in  slice $S_k$ is the
union of triangles $P_2(2^{k-1},t)$ and the sequence of their weights (from left to right) coincides
with $\beta(\bw_1)$. The top vertices of these triangles are on the first row of the slice and
their bases are adjacent and partition the bottom row.

\end{theorem}
\begin{proof}
     (1) follows by the definitions. (2) is proved in Lemma~\ref{LemmaW01}.
(3) The proof is by induction. The initial step, $k=1$, coincides with (2).

Suppose now that the stated description is valid for slice $S_k$ and let as look on $S_{k+1}$.
We begin with the first row of $S_{k+1}$.
Here, the first $2\cdot2^{k-1}-1$ cells are zeros, because, by the induction
hypothesis, on the previous row, the first $2\cdot2^{k-1}$ cells where the adjacent bases of two
triangles $P_2(2^{k-1},1)$.
The next element, the $2\cdot2^{k-1}$th, is $1=|1-2|$, since on the previous slice, the weight of
the second triangle was $1$ and the weight of the third triangle is $2$. Continuing in the same
way, we see that the next non-zero cell on the first row of  slice $S_{k+1}$ is the 
$4\cdot 2^{k-1}$th and its value is equal with $1=|2-1|$.
In this way we see that the non-zero cells on the first row of slice $S_{k+1}$ are those obtained
as absolute differences of the parent cells that are vertices of neighbor triangles
$P_2(2^{k-1},t)$ with different  weights from the previous slice. These are the 
$2^{k}$th, the $2\cdot2^{k}$th, the $3\cdot  2^{k}$th, 
and so on.
Moreover, by the induction hypothesis, the values of integers occupying these cells are the integers
in the sequence $\beta(\bw_1)$.

Now, using Lemma~\ref{LemmaP2t}, we find that from each of these cells, grows a triangle
$P_2(2^{k},t)$. Moreover, the weights of these triangles coincides with the integers on the non-zero
cells in the first row of $S_{k+1}$. Also, the size of the slice assures that the bases of these
triangles are adjacent. This concludes the proof of the induction step and of the theorem.
\end{proof}

In particular, Theorem~\ref{TheoremTNNsdoi} describes the western edge of the matrix $\TNNsdoi$.
\begin{corollary}\label{Corolarry1}
     The $2$-valuation of the elements of the west sequence $W_\NNs$ are:
\begin{equation*}
	\begin{split}
	 v_2\big(t(m, 1)\big) = 
		 \begin{cases} 
			1, & \   m = 2^k,\  k \geq 1 \\ 
			0, & \    \text{else.}
		 \end{cases}
	\end{split}
\end{equation*}
\end{corollary}
Also, Theorem~\ref{TheoremTNNsdoi} answers to a question of  
  N. J. A. Sloane, who
asks whether there is a proof that $4$ cannot appear on the western edge of the matrix
$T_\NNs$~\cite[A222313]{OEIS}.
\begin{corollary}\label{Corolarry2}
     There is no $4$ on $W(T_\NNs)$.
\end{corollary}

\section{The description of $T_\PP$}\label{SectionPP}
The powers of all primes in the decomposition of the terms in the sequence 
\begin{equation*}
\PP=\{1,2,3,2,5,6,7,2,3,10,11,6,13,14,15,2,17,6,19,\dots \}
\end{equation*}
defined in \eqref{eqDefPP} are equal to one. This allows us to employ operations in the the ring of
meromorphic series $\FF_2[[X]]$ to understand the structure of the $p$-tomographies of $T_\PP$.
Thus the initial generation of any $p$-tomography is sequence 
$\cA_p=\{\fp_n\}_{n\ge 1}$ defined by \eqref{eqAP1}. The superposition (component-wise
multiplication) of the $p$-tomographies for all $p$ gives the full description of  matrix
$T_\PP$.

For any prime $p$, we look at matrix $v_p(T_\PP)$. Always, the first row is filled with zeros
except the cells in the arithmetic progression $kp$, $k\ge 1$, which are equal to $1$.  Again, the
prime $p=2$ is exceptional. The second row of  $v_2(T_\PP)$ has all cells equal to $1$ and from the
third row on,  matrix $v_2(T_\PP)$ is filled with zeros only.

We show that if $p$ is odd, the rows can be grouped in periodic slices.
The number of rows in such
a slice is a \textit{period} of $v_p(T_\PP)$ and we denote by $\pi_p$ the length of the smallest
period. 
If $p=2$, the periodic slices contain just one row, which repeats from the third on.
If $p$ is odd, the first row of the first periodic slice is always the second row of 
$v_p(T_\PP)$.

One can check the small periods for some primes: $\pi_2=1$, $\pi_3=3$, $\pi_5=15$, $\pi_7=7$, 
$\pi_{31}=31$, $\pi_{127}=127$. As $p$ increases, the
size of $\pi_p$ becomes large: $\pi_{11}=341$, $\pi_{13}=819$, 
$\pi_{17}=255$, $\pi_{19}=9709$.
This fact produces the  general aspect of randomness of $v_p(T_\PP)$, for $p\ge 11$. 
Also, in some areas this phenomenon is more pronounced than in others  (see
Figure~\ref{FigTomo13-19}).

\begin{figure}[ht]
 \centering
 \mbox{
 \subfigure{
    \includegraphics[width=0.48\textwidth]{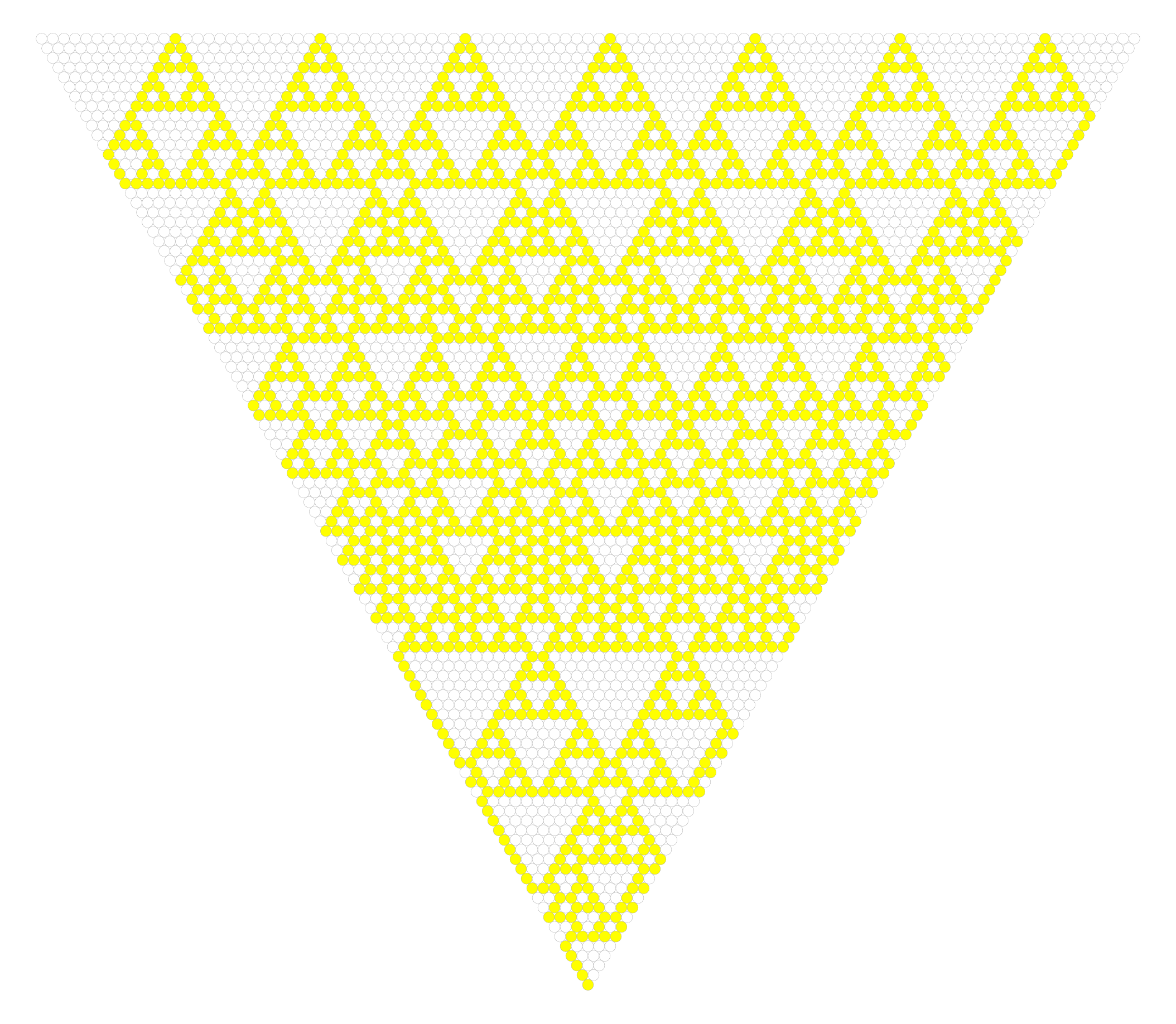}
 }
 \subfigure{
    \includegraphics[width=0.48\textwidth]{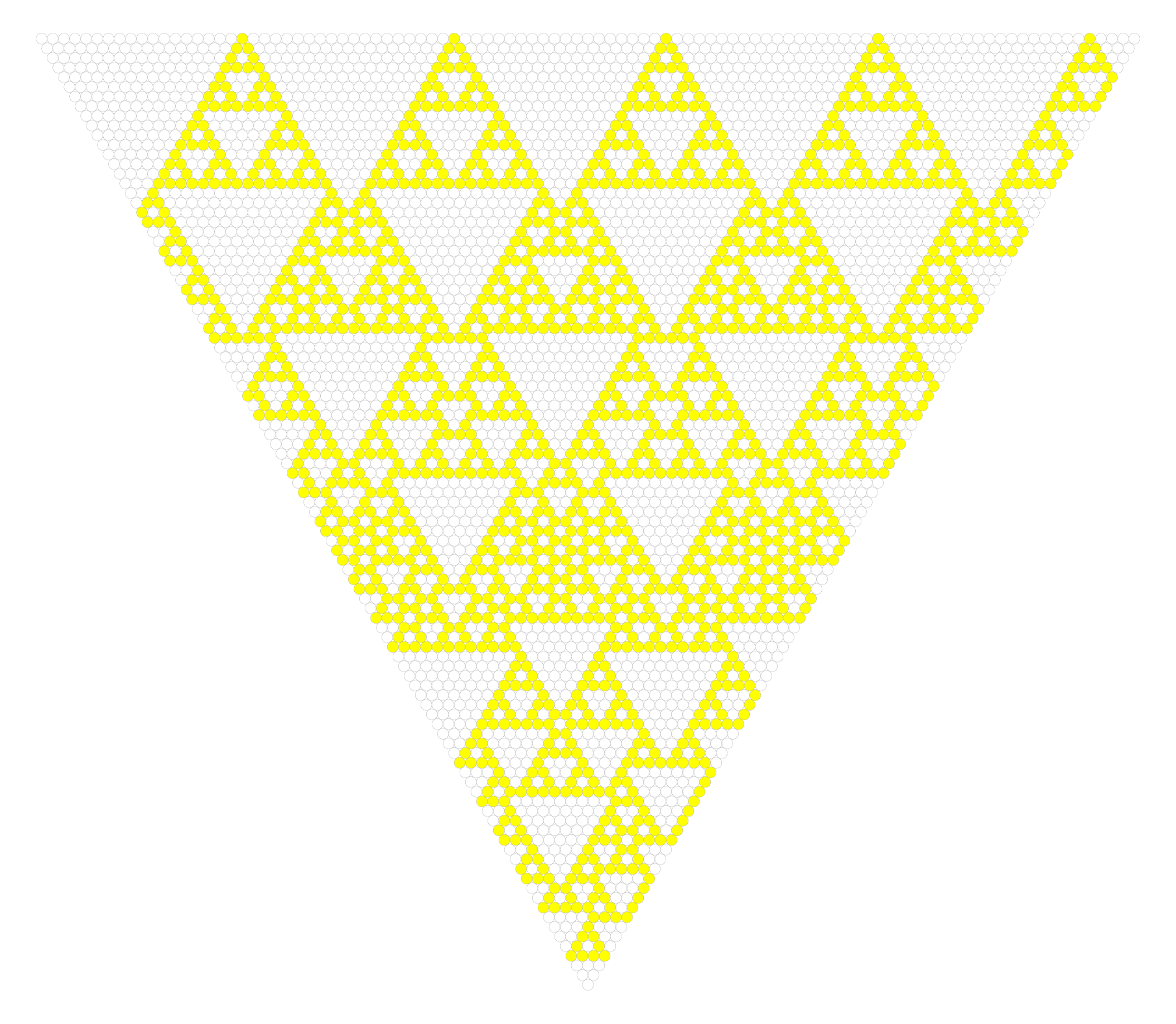}
 }
 }
\caption{The tomographies $v_p(T_{\PP}(100))$, for $p=13$ and $p=19$.
}
 \label{FigTomo13-19}
 \end{figure}

Any sequence $\be=\{e_k\}_{k\ge 1}\subset\{0,1\}^\NN$ can be identified uniquely with a series in
$\FF_2[[X]]^*$. We use this identifications for the rows of the matrix $v_p(T_{\PP,p})$ and write
\begin{equation*}
	\begin{split}
	\be=\{e_k\}_{k\ge 1}\stackrel{\theta}{\longleftrightarrow}
         \theta_\be(X)=\sum_{k\ge 1} e_kX^k\,.
	\end{split}
\end{equation*}
The operation of passing from one generation to the next by applying the
$Z(\cdot,\cdot)$-rule  \eqref{eqZrule} transfers on the side of the series to multiplication by
$\frac{1+X}{X}$. This may produce a series in $\FF_2[[X]]\setminus\FF_2[[X]]^{*}$ and we need to
bring it back by dropping the meromorphic part and the constant term through the $\Delta$ operation:
\begin{equation*}\label{eqDelta}
	\begin{split}
	\Delta(S(X)):=\sum_{k\ge 1}c_kX^k\in\FF_2[[X]]\,,\quad 
	       \text{for $S(X)=\sum  c_kX^k\in\FF_2[[X]]$}\,.
	\end{split}
\end{equation*}
Then, to pass from the $j$th generation to the $(j+m)$th, we have to multiply repeatedly $m$ times
by $\tfrac{1+X}{X}$, so the general correspondence is
\begin{equation}\label{eqCD}
\begin{split}
     \DiagramaComutativa
\end{split}
\raisetag{11mm}
\end{equation}
where $\alpha$ is the absolute value of the differences (which, in this case, coincides with
addition in $\FF_2$), taken component wise.

Next we show that this association is well defined.
\begin{proposition}\label{PropositionCommutes}
     The above association between the rows of the matrix $v_p(T_\PP)$ and the series in
$\FF_2[[X]]^*$ and the operation of passing from one generation to the next is well defined,
 and diagram \eqref{eqCD} is commutative.
\end{proposition}
\begin{proof}
Well defining is due to the correspondence between the absence of columns to the left of the
first column of $v_p(T_\PP)$, so there is no influence from the left when $\alpha$ is applied, and
from ignoring of the meromorphic and constant terms of the series using the dropping
function $\Delta$  .

It remains to  prove that diagram \eqref{eqCD} is commutative by induction.
The initial step: Suppose $\be_j=(e_1,e_2, e_3, \dots)$. Then, on the one hand, we have:
\begin{equation}\label{eqDR}
	\begin{split}
	\alpha(\be_{j})= \be_{j+1} & = (e_1+e_2, e_2+e_3, e_3+e_4\dots)\,,\\
	  \theta_{\be_{j+1}}(X) & =\sum_{k\ge 1}(e_k+e_{k+1})X^k,
	\end{split}
\end{equation}
and on the other hand
\begin{equation}\label{eqRD}
	\begin{split}
	       \theta_{\be_{j}}(X) & =\sum_{k\ge 1}e_k X^k,\\
	       \tfrac{1+X}{X}\cdot\sum_{k\ge 1}e_k X^k 
		    & = e_1 + \sum_{k\ge 1}(e_k+e_{k+1}) X^k,\\
	       \Delta\Big( e_1 + \sum_{k\ge 1}(e_k+e_{k+1}) X^k\Big) 
		    & = \sum_{k\ge 1}(e_k+e_{k+1})X^k\,.
	\end{split}
\end{equation}
The outcomes of \eqref{eqDR} and  \eqref{eqRD} are identical, so the initial step is completed.

The induction step follows by using the associative property of the composition of functions 
$\Delta$ and multiplication by $\tfrac{1+X}{X}$ and the fact that $\Delta^{(2)}=\Delta$. This
completes the proof of the proposition. 

\end{proof}

Let us see the periodicity of the matrix $v_p(T_\PP)$ in two particular cases.

\subsection{Periodicity of $v_3(T_\PP)$}
The first row of $v_3(T_\PP)$ contains only zeros, except the cells with ones in the columns with
ranks in the arithmetic progression $\{3n\}_{n\ge 1}$.
The series that corresponds to the second row is
\begin{equation}\label{eq3L2}
	\begin{split}
	\Delta\Big(\tfrac{1+X}{X}\cdot\sum_{k\ge 1}X^{3k}\Big)
	  &= (X^2+X^3)\sum_{k\ge 0}X^{3k}.
	\end{split}
\end{equation}
Then the series that corresponds to the $5$th row is
\begin{equation}\label{eq3L5}
	\begin{split}
	\Delta\Big(\left(\tfrac{1+X}{X}\right)^{5-2}(X^2+X^3)\sum_{k\ge 0}X^{3k}\Big)
	  &=	\Delta\Big((1+X+X^2+X^3)\left(\tfrac{1+X}{X}\right)\sum_{k\ge 0}X^{3k}\Big)\\
	  &=	\Delta\Big(\tfrac{1+X^4}{X}\cdot\sum_{k\ge 0}X^{3k}\Big)\\
	  &= (X^2+X^3)\sum_{k\ge 0}X^{3k}.
	\end{split}
\end{equation}
Comparing \eqref{eq3L2} and \eqref{eq3L5}, we see that the $2$nd and the $5$th rows coincide.
Therefore,
the matrix $v_3(T_\PP)$ is eventually periodic and $3$ is the length of a period.

\subsection{Periodicity of $v_5(T_\PP)$}
The series that corresponds to the second row of $v_5(T_\PP)$ is
\begin{equation}\label{eq5L2}
	\begin{split}
	\Delta\Big(\tfrac{1+X}{X}\cdot\sum_{k\ge 1}X^{5k}\Big)
	  &= (X^4+X^5)\sum_{k\ge 0}X^{5k}.
	\end{split}
\end{equation}
We take advantage of the fact that $15$ is a special number and all the binomial coefficients
$\binom{15}{k}$, $0\le k\le 15$, are odd. Then
\begin{equation*}
	\begin{split}
	\left(\tfrac{1+X}{X}\right)^{15}(X^4+X^5)\sum_{k\ge 0}X^{5k}
	  &=	(1+X+\cdots +X^{15})\cdot\tfrac{1+X}{X^{11}}\cdot\sum_{k\ge 0}X^{5k}\\
	  &=	\tfrac{1+X^{16}}{X^{11}}\cdot\sum_{k\ge 0}X^{5k}.
	\end{split}
\end{equation*}
Dropping the meromorphic and the constant term, we find that
\begin{equation*}
	\begin{split}
	\Delta\Big(\left(\tfrac{1+X}{X}\right)^{15}(X^4+X^5)\sum_{k\ge 0}X^{5k}\Big)
	  &= \Delta\Big(  \tfrac{1+X^{16}}{X^{11}}\cdot\sum_{k\ge 0}X^{5k}\Big)\\
	  &= (X^4+X^5)\sum_{k\ge 0}X^{5k}, 
	\end{split}
\end{equation*}
which, compared with \eqref{eq5L2} shows that the $2$nd row coincides with the $16$th. Thus
the matrix $v_5(T_\PP)$ is eventually periodic and $15$ is the length of a period. 
One can check that there is no shorter period. For this, it suffices to calculate the terms from
the first column of the matrix $v_5(T_\PP)$. The first $16$  of them are:
\begin{equation*}
	  v_5(W_\PP(16)) = 
\{ 0, \underbrace{0, 0, 0, 1, 1, 1, 1, 0, 1, 0, 1, 1, 0, 0,  1}_{\text{the west-period }}\}\,.
\end{equation*}

\subsection{Periodicity of $v_p(T_\PP), \ p$ odd}

For a general $p$, we can also take advantage of the fact that there are integers $M$ for which all 
the binomial coefficients $\binom{M}{k}$, $0\le k\le M$ are odd. Such integers do exist, as follows 
from the next simple lemma.

\begin{lemma}\label{LemmaBinomOdd}
For any integer $n\ge 1$, we have
\begin{equation*}
	\begin{split}
	v_2(u)=v_2(v), \  \text{for } 1\le u,v\le 2^{n-1},\ \text{with } u+v=2^n.
	\end{split}
\end{equation*}
\end{lemma}
Then, by Lemma~\ref{LemmaBinomOdd} and the definition of the binomial coefficients , we see that  
\begin{equation}\label{eqBinom1}
	\begin{split}
	\binom{2^n-1}{k}\equiv 1\pmod 2, \ \ \text{for } \ 0\le k\le n\,.
	\end{split}
\end{equation}

Another requirement for $M$ is to be divisible by $p$. 
A minimal value of $n$ for which $M=2^n-1$ is divisible by $p$ is $\ind_p(2)$. (We denote by 
$\ind_p(a)$ the smallest integer $1\le n\le p-1$ for which $a^n\equiv 1\pmod p$.)

Now let $M=dp$ for some integer $d\ge 1$.
The series associated to the second row of $v_p(T_\PP)$ is
\begin{equation}\label{eqS2}
	\begin{split}
	S_2(X)=(X^{p-1}+X^p)\sum_{k\ge 0}X^{pk}\,.
	\end{split}
\end{equation}
To get the series $S_{M+1}(X)$ corresponding to the $(M+1)$th row, we have to multiply $S_2(X)$ by 
$\tfrac{(1+X)^{dp}}{X^{dp}}$. First, let us see that
\begin{equation*}
	\begin{split}
	(X^{p-1}+X^p)\frac{(1+X)^{dp}}{X^{dp}} & = \frac{1+X}{X^{dp-p+1}}(1+X+\cdots+X^{dp})
		= \frac{1+X^{dp+1}}{X^{dp-p+1}}\,.
	\end{split}
\end{equation*}
Then
\begin{equation*}
	\begin{split}
	S_2(X)\cdot \frac{(1+X)^{dp}}{X^{dp}} 
	& = \frac{1+X^{dp+1}}{X^{dp-p+1}} \cdot \sum_{k\ge 0}X^{pk}\\
	& = 	\sum_{k\ge 0}X^{pk-dp+p-1}	+ \sum_{k\ge 0}X^{pk+p}	.
	\end{split}
\end{equation*}
Here the meromorphic and constant terms occur only in the first sum. Dropping them, we arrive at
\begin{equation*}
	\begin{split}
	S_{M+1}(X) = \Delta\Big( \sum_{k\ge 0}X^{pk-dp+p-1}	+ \sum_{k\ge 0}X^{pk+p}\Big)
		= S_2(X)\,.
	\end{split}
\end{equation*}
In conclusion, we have proved the following theorem.

\begin{theorem}\label{TheoremPeriodicity}
     For any prime $p\ge 3$, the rows of the matrix $v_p(T_\PP)$ are eventually periodic. The 
pre-period contains only the first row of the matrix and the length of the smallest period is a 
divisor of $2^{\ind_p(2)}-1$.
\end{theorem}

We remark that $2^{\ind_p(2)}-1$ is not always the size of the smallest period. For example, if 
$p=11$,
$\ind_{11}(2)=10$ and $2^{10}-1=1023=3\cdot 11\cdot 31$, but  the length of the smallest period is 
$\pi_{11}=11\cdot 31=341$.
 Also, if $p=13$, $\ind_{13}(2)=12$ and $2^{12}-1=4095=3^2\cdot 5\cdot 7\cdot 13$, but  the length 
of the 
smallest period is 
$\pi_{13}=3^2\cdot 7\cdot 13=819$.
As well, if $p=19$, $\ind_{19}(2)=18$ and $2^{18}-1=262143=3^3\cdot 7\cdot 19\cdot 73$. In 
this case, again,  the length of the smallest period is  shorter, 
$\pi_{19}=(2^{18}-1)/3^3=9709$.

\medskip 

There are two classes of primes: the first one, for which the length of the period of 
$v_p(T_\PP)$ is maximal (that is, $2^{\ind_p(2)}-1$) and the second one,
for which the length of the period is strictly smaller than $2^{\ind_p(2)}-1$.
We do not know if either one or both of these classes contain infinitely many 
primes.

The reason for the shorter periods in these cases are the arithmetic properties that produce
favorable patterns in the series of binomial coefficients. Thus, writing the binomial coefficients 
\mbox{$\binom{H}{k}\pmod 2$,} $0\le k\le H$, as concatenated letters of a word and the repeated 
letters as 
powers, for 
$H=341$, they are:
	\begin{equation*}
	\begin{split}
%
%
%
1^20^21^20^{10}1^20^21^20^{42}1^20^21^20^{10}1^20^21^20^{170}1^20^21^20^{10}1^20^21^20^{42} 
1^20^21^20^{10}1^20^21^2
	\end{split}
\end{equation*}
and for $H=819$, they are:
\begin{equation*}
	\begin{split}
	1^40^{12}1^40^{12}1^40^{12}1^40^{204}
1^40^{12}1^40^{12}1^40^{12}1^40^{204}
1^40^{12}1^40^{12}1^40^{12}1^40^{204}
1^40^{12}1^40^{12}1^40^{12}1^4.
	\end{split}
\end{equation*}
A related  pattern appears if $p=19$.
In this case $H= 9709$ 
and the word defined by the  binomial coefficients is 
\begin{equation*}
	\begin{split}
	(AB)^{16} 
0^{512} 
(AB)^{16}
0^{6638}
(BA)^{16}
0^{512}
(BA)^{16},
	\end{split}
\end{equation*}
where
$A=1^20^2 1^20^2 1^2 0^2 1^2$ and
$B=0^{18}$.

\section{Extreme values on the West Side of $W_\PP$}
Calculations using power series from $\FF_2[[X]]$ allows us to quickly find a particular element 
of  matrix $T_\PP$. In particular, we can find the ``extreme values'' of $W_\PP(m)$, $m\ge 1$. 
They emerge on the $m$th row of $T_\PP$ in 
places where number $m$, when written in base two, has either few or many ones, compared with the 
rank of rows in its neighborhood.
One can notice this property in the augmented oscillations of both graphs in 
Figure~\ref{FigLoP}.
In the following, we present the concrete structure of the most pronounced extremes, the values of
$W_\PP(m)$, with $m$ around powers of two.
We have to consider only the influence of odd primes, since $p=2$ is involved only on the first 
two rows of $T_\PP$.

\subsection{The size of $W_\PP(2^g)$}\label{Subsection1}
Let $m=2^g-1$. By \eqref{eqBinom1} we know that $\binom{m}{j}\equiv 1\pmod 2$, for $0\le j\le m$.
To find the series associated to the $(m+1)$th row of the  $p$-topography of $T_\PP$, we have 
to multiply:
\begin{equation}\label{eqhA}
	\begin{split}
	\left(\tfrac{1+X}{X}\right)^m\sum_{k\ge 1}X^{kp}
	=& \tfrac{1}{X^m}(1+X+\cdots+X^m)\sum_{k\ge 1}X^{kp}\\
	=&	(X^{p-m}+X^{p-m+1}+\cdots +X^{p}) + (X^{2p-m}+X^{2p-m+1}+\cdots +X^{2p}) +\\
		& + (X^{3p-m}+X^{3p-m+1}+\cdots +X^{3p}) +\cdots
	\end{split}
\end{equation}
Then, $p$  divides $W_\PP(2^g)$ if and only if the coefficient of $X$ in series \eqref{eqhA} is 
odd. We see that primes $p\ge m+2$ are not involved and $p=m+1$ is impossible.

For any small primes $3\le p\le m$, denote by $\lambda=\lambda(p,m)\ge 1$ the largest integer for 
which there exist integers $0\le s_1, s_2,\dots, s_{\lambda}\le m$, such that
\begin{equation}\label{eqDef_Lp}
	\begin{split}
	1&=p+s_1-m,\\
	1&=2p+s_2-m,\\[-0.65em]
	  &\ \, \vdots\\[-0.35em]
	1&=\lambda p+s_{\lambda}-m.\\
	\end{split}
\end{equation}
Note that $\lambda(p,m)$ exists and $\lambda(p,m)\le (m+1)/p$.
Then,  monomial $X$ appears in the $\FF_2[[X]]$ series \eqref{eqhA} if and only if 
$\lambda(p,m)$  
is odd.

For example, if $m=15$, by a simple investigation we find that $\lambda(7,15)$ is even and  
$\lambda(p,15)$ is odd for $p=3, 5, 11, 13$, so $W_\PP(16)=3\cdot 5\cdot 11\cdot 13=2145$.
In the same way, if $m=31$, we see that $\lambda(p,31)$ is 
even only for $p=3,5,7,11,13$, 
so $W_\PP(32)=17\cdot 19\cdot 23\cdot 29\cdot 31=6678671$.
\subsection{The maximum $W_\PP(2^g-1)$}\label{Subsection2}
Let $m=2^g-2$ with $g\ge 2$. Then, in $\FF_2[[X]]$ we have
\begin{equation*}
	\begin{split}
	(1+X)^m=1+X^2+X^4+\cdots+X^m\,.
	\end{split}
\end{equation*}
Finding the series that corresponds to the $(2^g-1)$th row of $T_\PP$ implicates the calculation:
\begin{equation}\label{eqhB}
	\begin{split}
	\left(\tfrac{1+X}{X}\right)^m\sum_{k\ge 1}X^{kp}
	=& \tfrac{1}{X^m}(1+X^2+\cdots+X^m)\sum_{k\ge 1}X^{kp}\\
	=&	(X^{p-m}+X^{p-m+2}+\cdots +X^{p}) + (X^{2p-m}+X^{2p-m+2}+\cdots +X^{2p}) +\\
		& + (X^{3p-m}+X^{3p-m+2}+\cdots +X^{3p}) +\cdots
	\end{split}
\end{equation}	
Again, we have to look for terms whose power of $X$  is equal to one.
In series \eqref{eqhB}, the terms corresponding to primes $p\ge m+2$ do not contribute to the 
coefficient of $X$. Also, $p$ can not be equal to $m$, because $m$ is even. 

If $2^g-1$ is a Mersenne prime, then $p=m+1$ is equal with this prime. Then $X^{p-m}=X$, so $p$  
divides $W_\PP(2^g-1)$.

For the remaining primes
$3\le p < m$, let $\mu=\mu(p,m)$ be the the maximal number of equalities 
\begin{equation}\label{eqDef_Mp}
	\begin{split}
	1&=p+t_1-m,\\
	1&=2p+t_2-m,\\[-0.65em]
	  &\ \, \vdots\\[-0.35em]
	1&=\mu p+t_{\mu}-m,
	\end{split}
\end{equation}
where $t_1, t_2,\dots, t_\mu$ are even numbers that belong to $\{0,2,\dots,m\}$.
Notice that $\mu(p,m)\le (m+1)/p$.
Then monomial $X$ effectively appears in  series \eqref{eqhB} if and only if 
$\mu(p,m)\equiv 1\pmod 2$. Therefore $p \mid W_\PP(2^g-1)$ if and only if $\mu(p,m)$ is odd.

Examples:
If $m=14$, we find that 
$\mu(3,14)=3$; $\mu(5,14)=2$; and $\mu(7,14)=\mu(11,14)=\mu(13,14)=1$, so 
$W_\PP(15)= 3\cdot 7\cdot 11\cdot 13 =3003$.

If  $m=30$, $p=m+1$ is a Mersenne prime. For the smaller primes, we find that 
$\mu(3,30)=5$; $\mu(5,30)=3$; $\mu(7,30)=2$; and
$\mu(11,30)=\mu(13,30)=\mu(19,30)=\mu(23,30)=\mu(29,30)=1$.
This implies that 
$W_\PP(31)= 3\cdot 5 \cdot11\cdot 13 \cdot 17  \cdot 19 \cdot 23 \cdot 29 \cdot 31=14325749295$.

\subsection{The minimum $W_\PP(2^g+1)$}\label{Subsection3}
Let $m=2^g$.
The smaller numbers on the west side of $T_\PP$ appear on the rows of rank $2^g+1$. This is due to 
the fact that $(1+X)^m=1+X^m$ in $\FF_2[[X]]$, that is, the binomial $(1+X)^m$  has fewest possible 
terms. Then, the series that correspond to the $(2^g+1)$th row sums the terms of positive powers 
of $X$ from the following
\begin{equation}\label{eqhC}
	\begin{split}
	\left(\tfrac{1+X}{X}\right)^m\sum_{k\ge 1}X^{kp}
	=& \tfrac{1}{X^m}(1+X^m)\sum_{k\ge 1}X^{kp}\\
	=&	(X^{p-m}+X^{p}) +(X^{2p-m}+X^{2p}) + (X^{3p-m}+X^{3p}) + \cdots
	\end{split}
\end{equation}	
For a given $p\ge 3$, on the right-hand side of \eqref{eqhC} may appear a single monomial $X$, 
and this happens 
whenever there exists an integer $d\ge 1$, such that $dp - m = 1$. 
This implies that the only prime divisors $p$ of  $W_\PP(2^g+1)$ are those for which 
if $p \mid (m+1)$.

For example, $W_\PP(9)=3$;  $W_\PP(17)=17$; $W_\PP(33)=33$ and 
$W_\PP(1025)=5\cdot 41 = 205$ (because $1025=5^2\cdot 41$); 
$W_\PP(32769)=3\cdot 11 \cdot 331 = 10923$ (because $32769=2^{15}+1=3^2\cdot 11\cdot 331$).

\medskip\medskip

Other terms of sequence $W_\PP$ may be calculated in the same way. A few more examples are 
listed in Table~\ref{TableWest2lag}.

\begin{table}[!htbp]
\small\tiny
\centering
\caption{The size of $W_\PP(m)$ for $m$ around $2^g$}\label{TableWest2lag}
$
\begin{array}{*5c}
\toprule
\text{power} &  m &  W_\PP(m) & \text{decomposition of }W_\PP(m) & \omega(W_\PP(m))\\
	\midrule
\multirow{5}{*}{$g=6$}	
   & 62   & 3.49 \cdot 10^{9}   & 23 \cdot 31 \cdot 37 \cdot 41 \cdot 53  \cdot 61 & 6\\
   &  63 & 2.79\cdot 10^{18}
  &  3 \cdot 7 \cdot 11 \cdots 59 \cdot 61  & 13\\
   &  64 & 4.36\cdot 10^{16}  & 3 \cdot 7 \cdot 11 \cdots 59 \cdot 61   & 12\\	
   &  65 & 65   & 5 \cdot 13  & 2\\	
   &  66 &  2145  & 3 \cdot 5 \cdot 11  \cdot 13 & 4\\
			\midrule
\multirow{5}{*}{$g=7$}	
   & 126   & 2.42\cdot 10^{21}
	& 3 \cdot 5 \cdot 7 \cdots 109 \cdot 113  & 14\\
   &  127 & 7.87\cdot 10^{39}  & 3 \cdot 5 \cdot 7 \cdots 113 \cdot 127  & 24\\
   &  128 & 1.45\cdot 10^{34}  & 5 \cdot 11 \cdot 13 \cdots 113 \cdot 127   & 20\\	
   &  129 & 129   & 3 \cdot 43  & 2\\	
   &  130 &  8385  & 3 \cdot 5 \cdot 13 \cdot 43   & 4\\
			\midrule
\multirow{5}{*}{$g=8$}	
   & 254   & 6.86\cdot 10^{28} & 103 \cdot 107  \cdot 127 \cdots 233 \cdot 241  & 13\\
   &  255 & 4.20\cdot 10^{76}  & 3 \cdot 19 \cdot 37 \cdots 241 \cdot 251  & 37\\
   &  256 & 1.17\cdot 10^{72}  & 3 \cdot 5 \cdot 11 \cdots 241 \cdot 251   & 37\\	
   &  257 & 257   & 257   & 1\\	
   &  258 &  33153  &  3 \cdot 43 \cdot 257   & 3\\
					\midrule
\multirow{5}{*}{$g=9$}	
   & 510   & 5.17\cdot 10^{92} & 3 \cdot 11  \cdot 19 \cdots 461 \cdot 509  & 42\\
   &  511 & 4.35\cdot 10^{168}  & 3 \cdot 5 \cdot 7 \cdots 503 \cdot 509  & 74\\
   &  512 & 8.03\cdot 10^{147}  & 7 \cdot 13 \cdot 29 \cdots 503 \cdot 509   & 63\\	
   &  513 & 57   & 3\cdot 19   & 2\\	
   &  514 &  14649  & 3 \cdot 19 \cdot 257   & 3\\
							\midrule
\multirow{5}{*}{$g=10$}	
   & 1022   & 9.32\cdot 10^{173} & 7 \cdot 71  \cdot 109 \cdots 1013 \cdot 1021  & 65 \\
   &  1023 & 2.53\cdot 10^{344}  & 3 \cdot 7 \cdot 11 \cdots 1019 \cdot 1021  & 132 \\
   &  1024 & 4.72\cdot 10^{298}  & 3 \cdot 11 \cdot 19 \cdots 1019 \cdot 1021   & 115 \\	
   &  1025 & 205   & 5\cdot 41   & 2\\	
   &  1026 &  11685  & 3 \cdot 5 \cdot 19  \cdot 41  & 4\\
\bottomrule
\end{array}
$
\end{table} 
 
\medskip
We collect the results from Sections~\ref{Subsection1}-\ref{Subsection3} into the next theorem.

 \medskip\medskip \medskip
 
\begin{theorem}\label{TheoremWPP2lap}
     Let $g\ge 2$ and let $\lambda=\lambda(p,2^g-1)$ and $\mu=\mu(p,2^g-2)$ be the integers defined 
by \eqref{eqDef_Lp} and \eqref{eqDef_Mp}.
Then

\begin{equation*}
	\begin{split}
	W_\PP(2^g-1)=\prod_{\substack{0\le p \le 2^g-3\\ \mu(p,2^g-2) \text{ odd}}}p;\quad
	W_\PP(2^g)=\prod_{\substack{0\le p \le 2^g-1\\ \lambda(p,2^g-1) \text{ odd}}}p;\quad	
	W_\PP(2^g+1)=\prod_{p \mid 2^g+1}p\,.\ \ 		
	\end{split}
\end{equation*}

\end{theorem}

 \medskip
 
\begin{paradoxproblem}\label{Paradox1}
	Explain why $W_\PP(2^g-1)$ is larger than $W_\PP(2^g)$,
		even if in the definition of $\mu(p,m)$, in equalities \eqref{eqDef_Mp}, an extra parity 
restriction on numbers $t_j$ is imposed (condition that is absent for the existence of numbers 
$s_j$ in \eqref{eqDef_Lp}).   
\end{paradoxproblem}

\medskip\medskip

Removing duplicates and ordering $W_\PP$, we obtain  sequence
\begin{equation}\label{eqSortedPP}
	\begin{split}
UO(W_\PP) : \ &1, 2, 3, 5, 11, 15, 17, 33, 35, 51, 57, 65, 91, 105, 129, 165, 195, 205, 221, \\
	&255, 257, 385, 451, 561, 861, 897, 969, 1615,\dots
	\end{split}
\end{equation}
This is related and has terms close to those of the analogues sequence~\cite[A222313]{OEIS}, 
\cite[Question 3]{CZ2014}
obtained by starting with the initial generation $\NNs$ 
instead of $\PP$. A complete discussion based on the previous analysis might give a complete 
argument for the certainty of the ranks of terms in list \eqref{eqSortedPP}.

\section{The West-Side of $T_\PP$ and $T_\NNs$}\label{SectionWestPP}
By Theorem~\ref{TheoremPeriodicity} it follows that sequence $v_p(W_\PP)$, the west edge of the 
matrix $v_p(T_\PP)$, is also periodic, for any odd prime and the pre-period contains only the 
first term of the sequence. We do not know whether there is a prime $p$ for which the length of 
the 
period of $v_p(W_\PP)$ is strictly smaller than the length of the period of $v_p(T_\PP)$. 
If there is such a prime, then it should be larger than $23$.

Comparing the general aspect of the $p$-tomographies of $T_\NNs$ and $T_\PP$, one can observe both 
similarities and significant differences. 
Thus, on the one hand, although there are more and more irregularities in $v_p(T_\PP)$ as $p$ 
increases, it is still eventually periodic. On the other hand, a big noise grows under the cells 
with 
larger and larger powers of $p$, if the initial generation is $\NNs$.
The most noticeable difference is if $p = 2$, since $v_2(T_\PP)$ has only zero-cells from the third 
row on, while $v_2(T_\NNs)$ sprouts the triangles in Figure~\ref{FigTomo2}. For small powers of 
$p=3$ and $p=5$, the results are shown side by side in Figures~\ref{FigTomo3} and \ref{FigTomo5}.
\begin{figure}[ht]
 \centering
 \mbox{
 \subfigure{
    \includegraphics[width=0.48\textwidth]{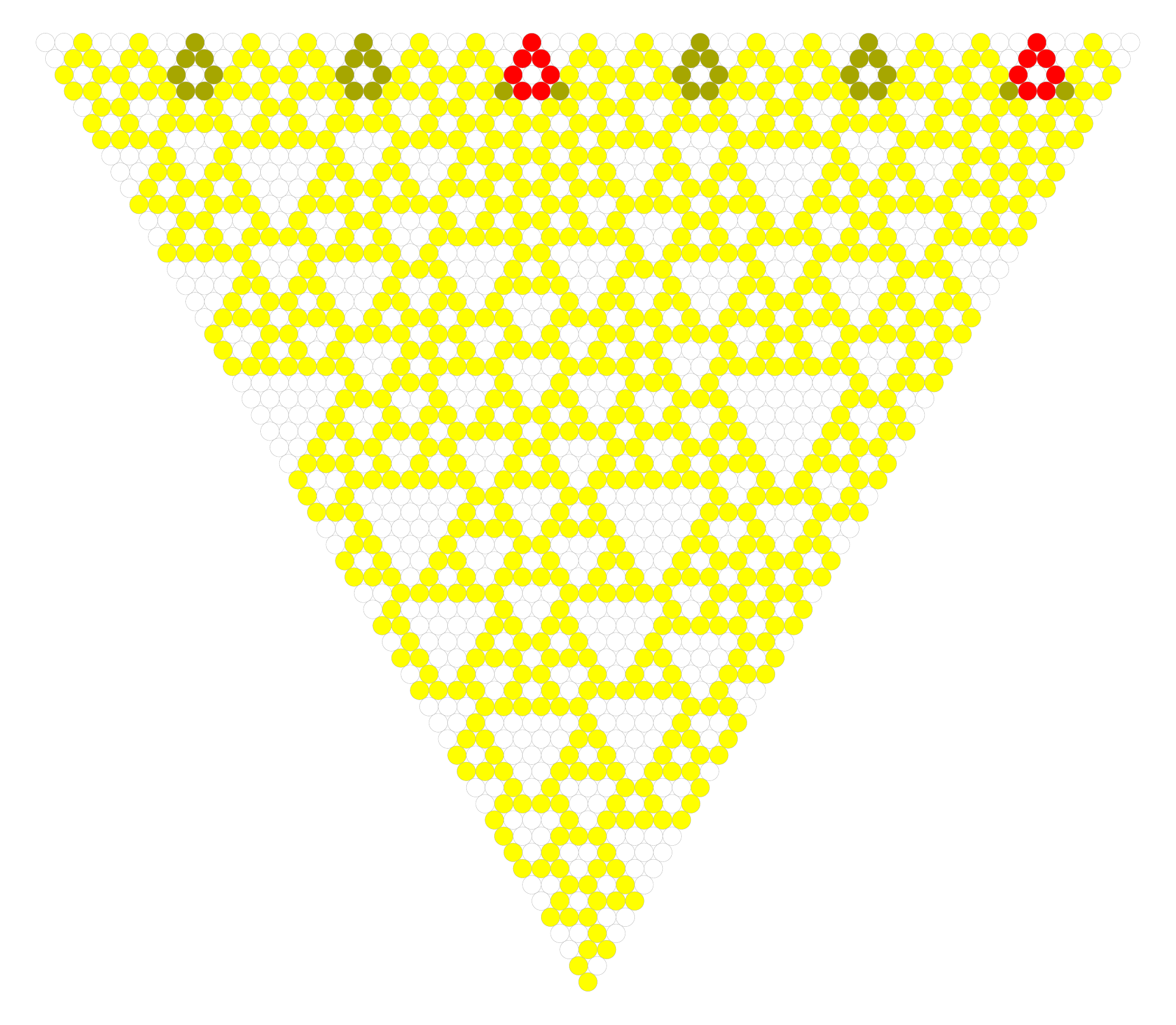}
 \label{FigTomo3N}
 }
 \subfigure{
    \includegraphics[width=0.48\textwidth]{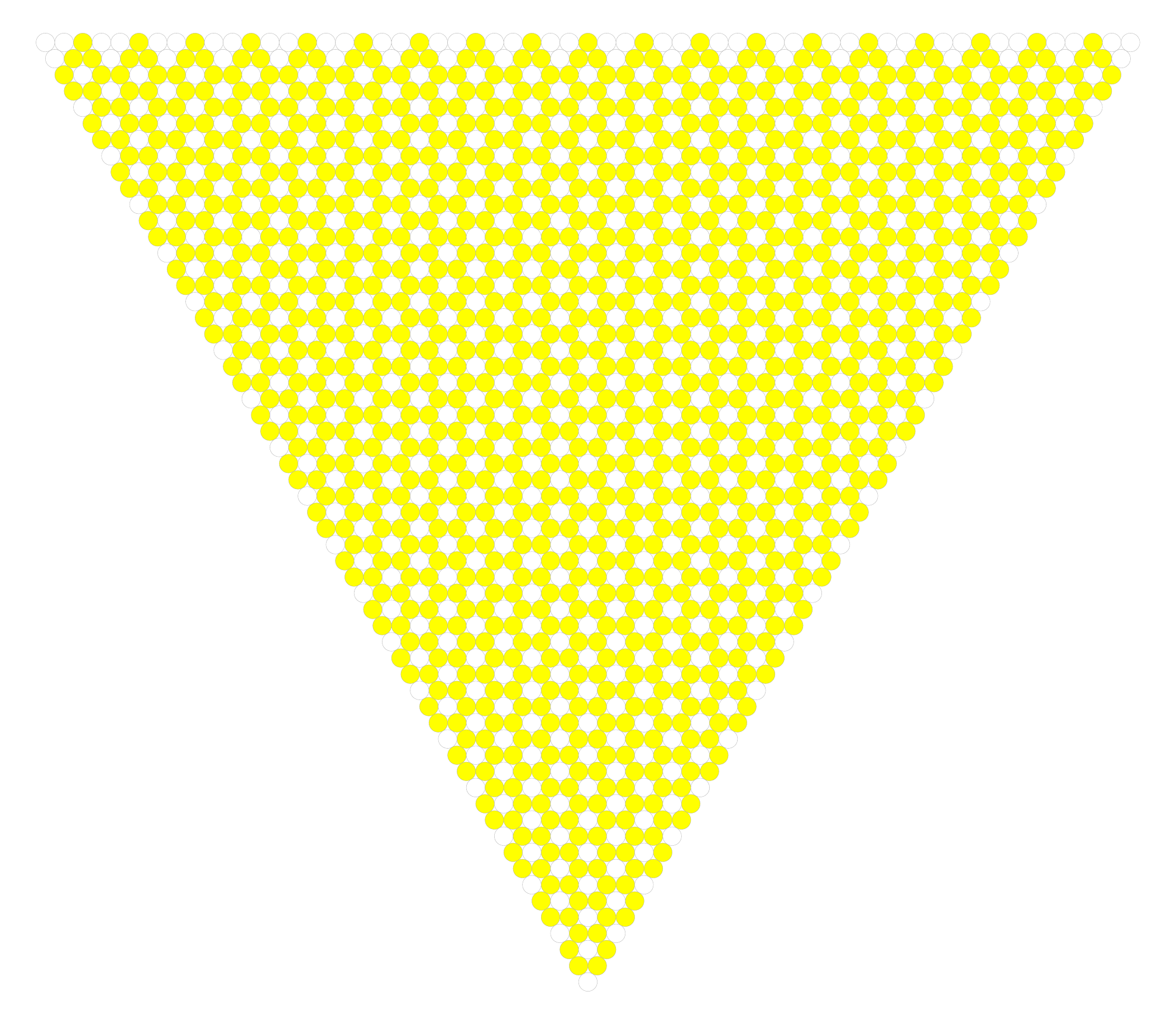}
 \label{FigTomo3P}
 }
 }
 \caption{The tomographies of  $T_{\NNs}(60)$ (left) and $T_{\PP}(60)$ (right), for $p=3$.
	  \vspace*{-.11em}  \\
     	  \mbox{\small Code of colors for the cells containing the powers of $p$, 
from $0,1,2,3$: \protect \Legend_0-3}	
}
 \label{FigTomo3}
 \end{figure}

\begin{figure}[ht]
 \centering
 \mbox{
 \subfigure{
    \includegraphics[width=0.48\textwidth]{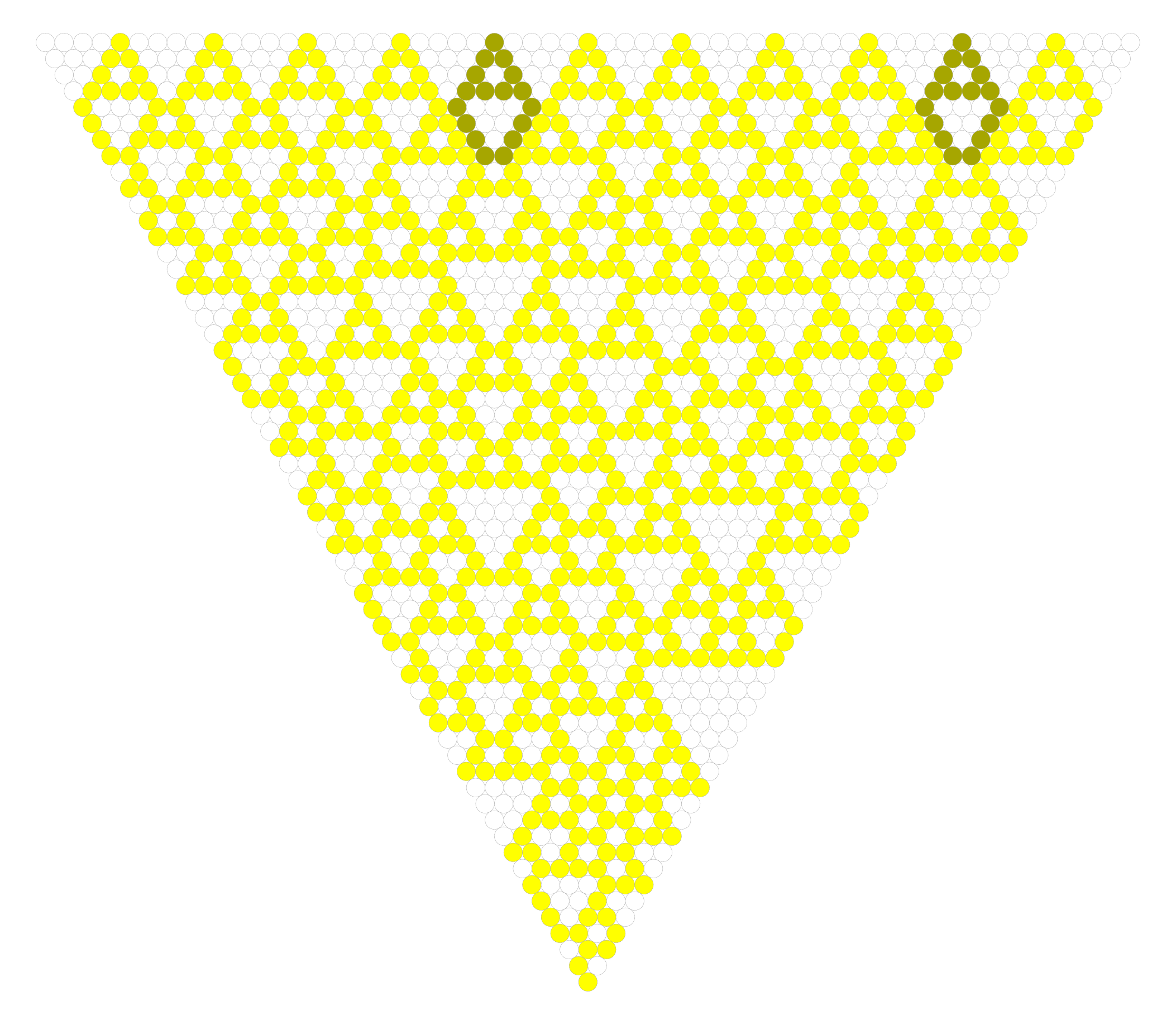}
 \label{FigTomo5N}
 }
 \subfigure{
    \includegraphics[width=0.48\textwidth]{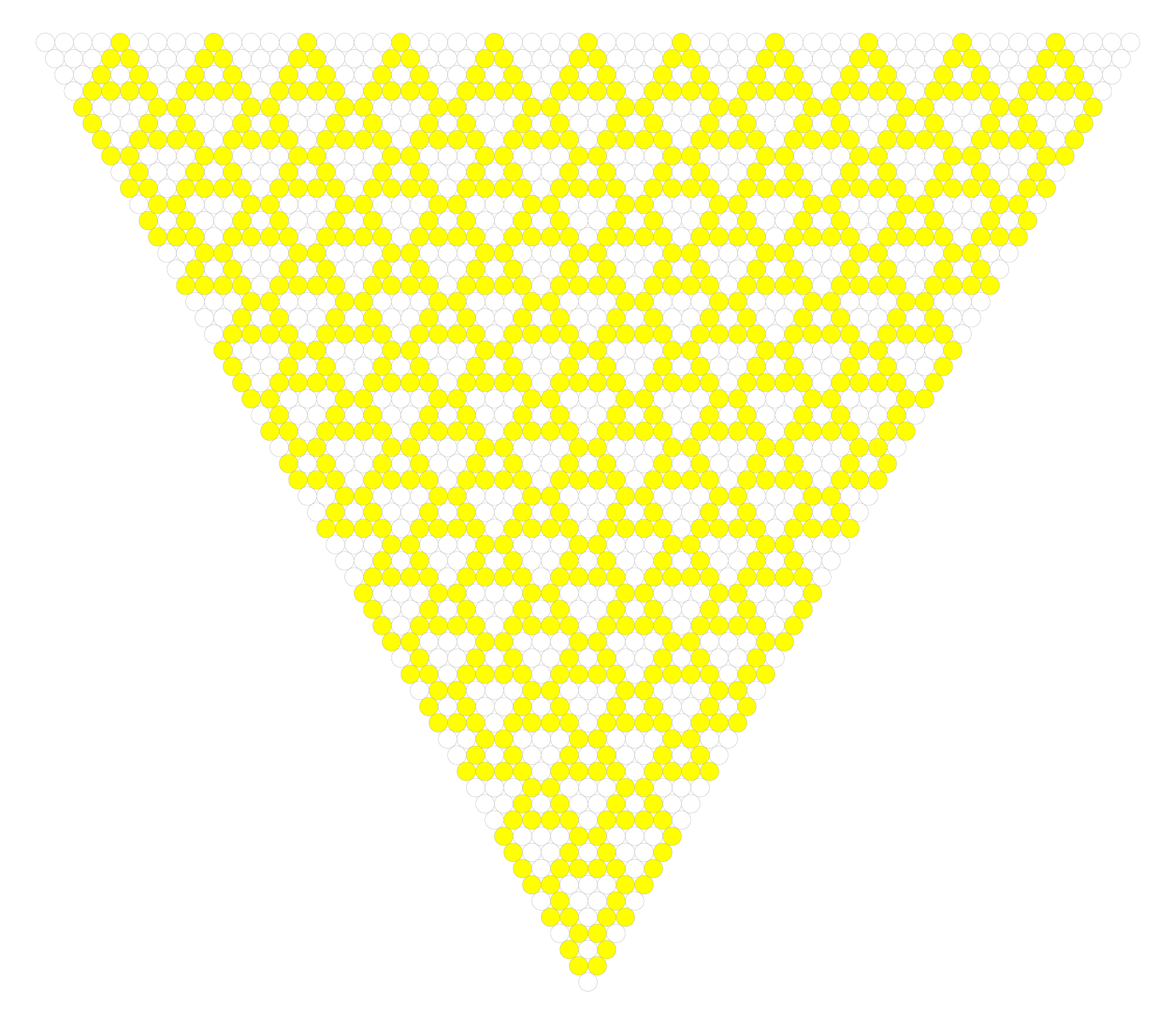}
 \label{FigTomo5P}
 }
 }
\caption{The tomographies of $T_{\NNs}(60)$ (left) and $T_{\PP}(60)$ (right) for $p=5$.
	  \vspace*{-.11em}  \\
     	  \mbox{\small Code of colors for the cells containing the powers of $p$, 
from $0,1,2,3$: \protect \Legend_0-3}	
}
 \label{FigTomo5}
 \end{figure}

Although the noise is transmitted till the west edge, it does not cover it completely. 
The beginning of sequences $W_\NNs$ and $W_\PP$ are:
{\small
\begin{equation*}
     \begin{split}
       W_{\NNs}(35)= \{&1, 2, 3, 6, 5, 15, 105, 70, 1, 5, 33, 55, 65, 273, 1001,1430, 17, 17,\\ 
		& 969, 4845, 1785, 6545, 37145, 81719, 17,1105, 3553,\\ 
                & 969969, 672945, 81345, 955049953, 66786710, 33, 561, 385\}
    \end{split}
\end{equation*}
}
and
{\small
\begin{equation*}
     \begin{split}
       W_{\PP}(35)= \{&1, 2, 3, 3, 5, 15, 105, 35, 3, 15, 11, 165, 195, 91, 3003, 2145, 17, 51,\\ 
		& 969, 1615, 1785, 19635, 
37145, 245157, 255, 221, 53295,\\ 
                &  4849845, 44863, 16269, 14325749295, 6678671, 33, 561, 385\}\,.
    \end{split}
\end{equation*}
}
They are equal in $13$ places, at terms of indices $1, 2, 3, 5, 6, 7, 17, 19, 21, 23, 33, 34, 35$.
As far as we can check, this semblance remains valid, suggesting a general behavior.
Even in places where they differ, the terms are very close, both in size and in arithmetic 
structure. 
Compare Figures~\ref{FigLoN} and \ref{FigLoP} to see more similarities of sequences 
$W_\NN$ and $W_\PP$.

\begin{figure}[ht]
 \centering
 \mbox{
 \subfigure{
    \includegraphics[width=0.48\textwidth]{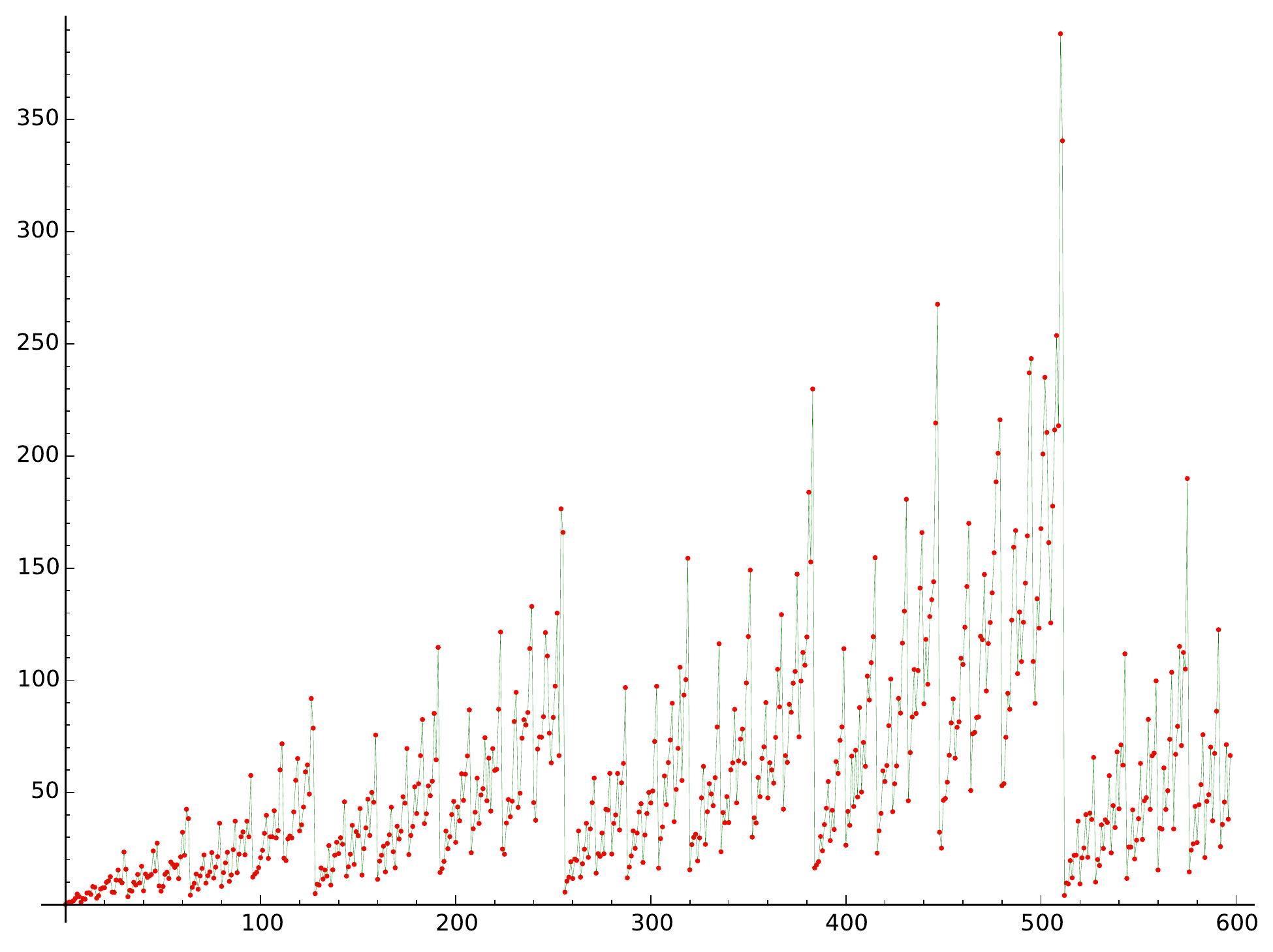}
 \label{FigLogP}
 }
 \subfigure{
    \includegraphics[width=0.48\textwidth]{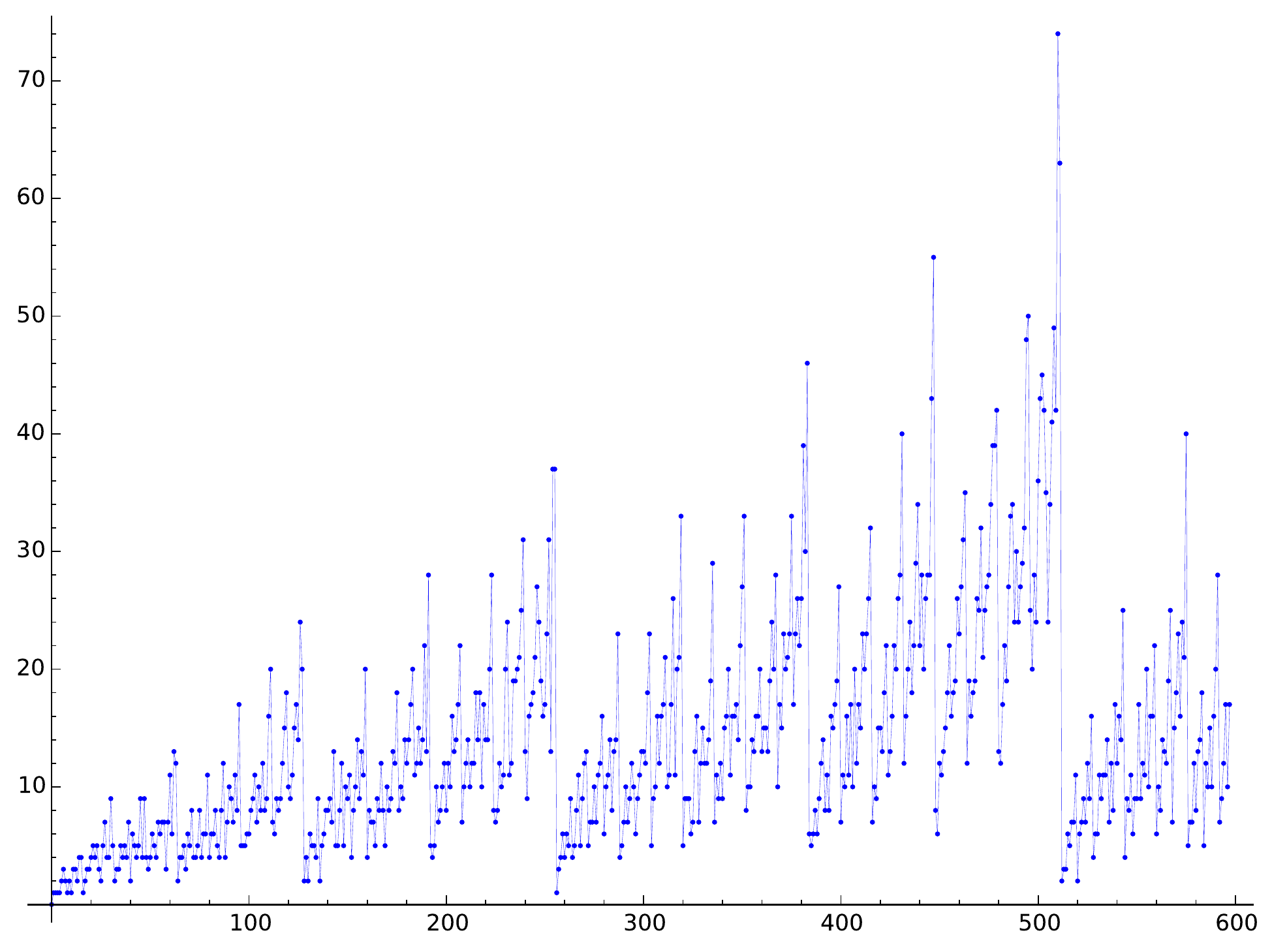}
 \label{FigomegaP}
 }
 }
 \caption{Comparison of the size and structure of the sequence $W_{\PP}(m), m\ge 1$:\\
     \texttt{Left:} the graph of $\log(W_{\PP}(m))$; \ \ \texttt{Right:} the graph of
$\omega(W_{\PP}(m))$. 
}
 \label{FigLoP}
 \end{figure}
\begin{figure}[ht]
 \centering
 \mbox{
 \subfigure{
    \includegraphics[width=0.48\textwidth]{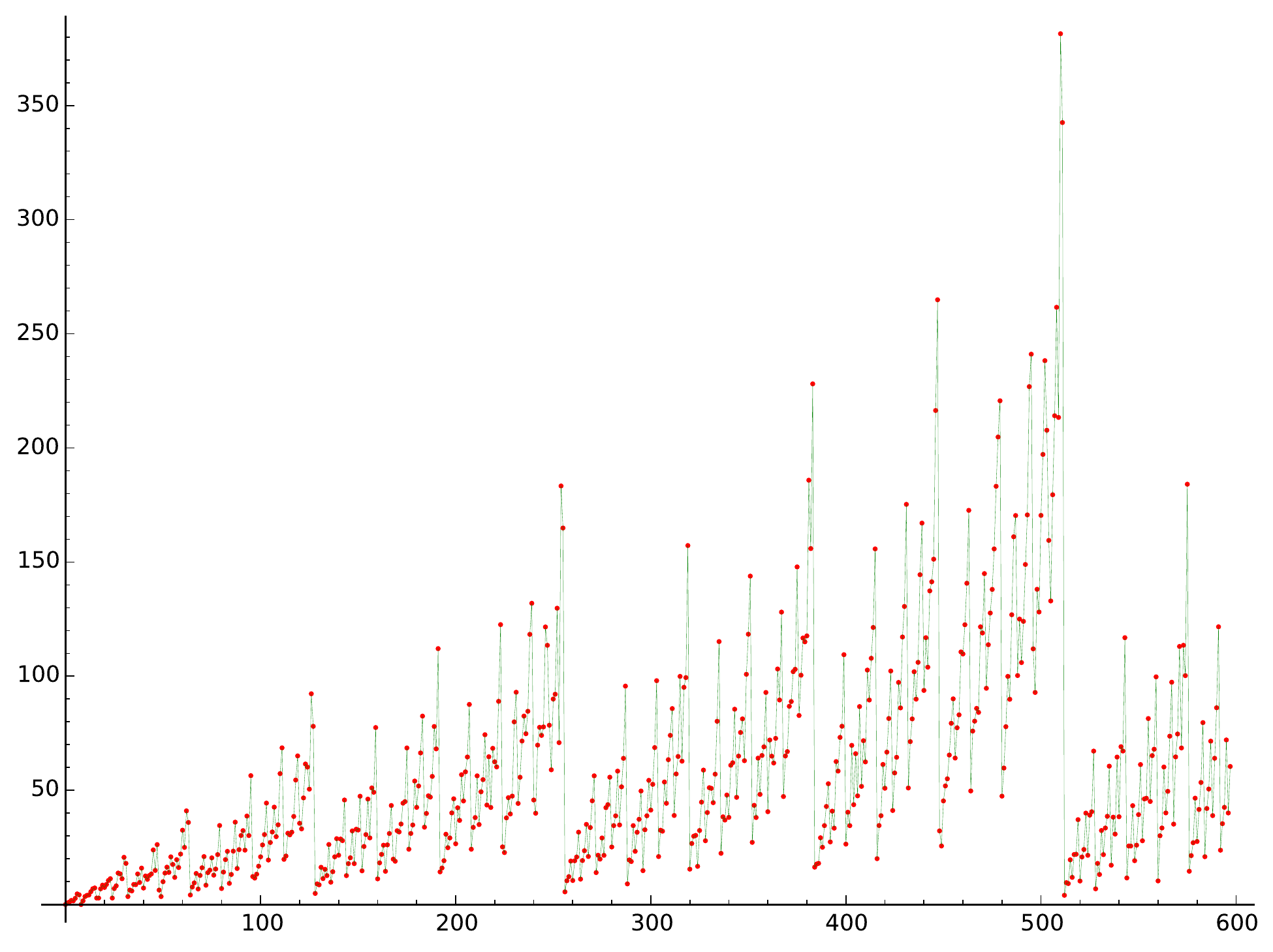}
 \label{FigLogN}
 }
 \subfigure{
    \includegraphics[width=0.48\textwidth]{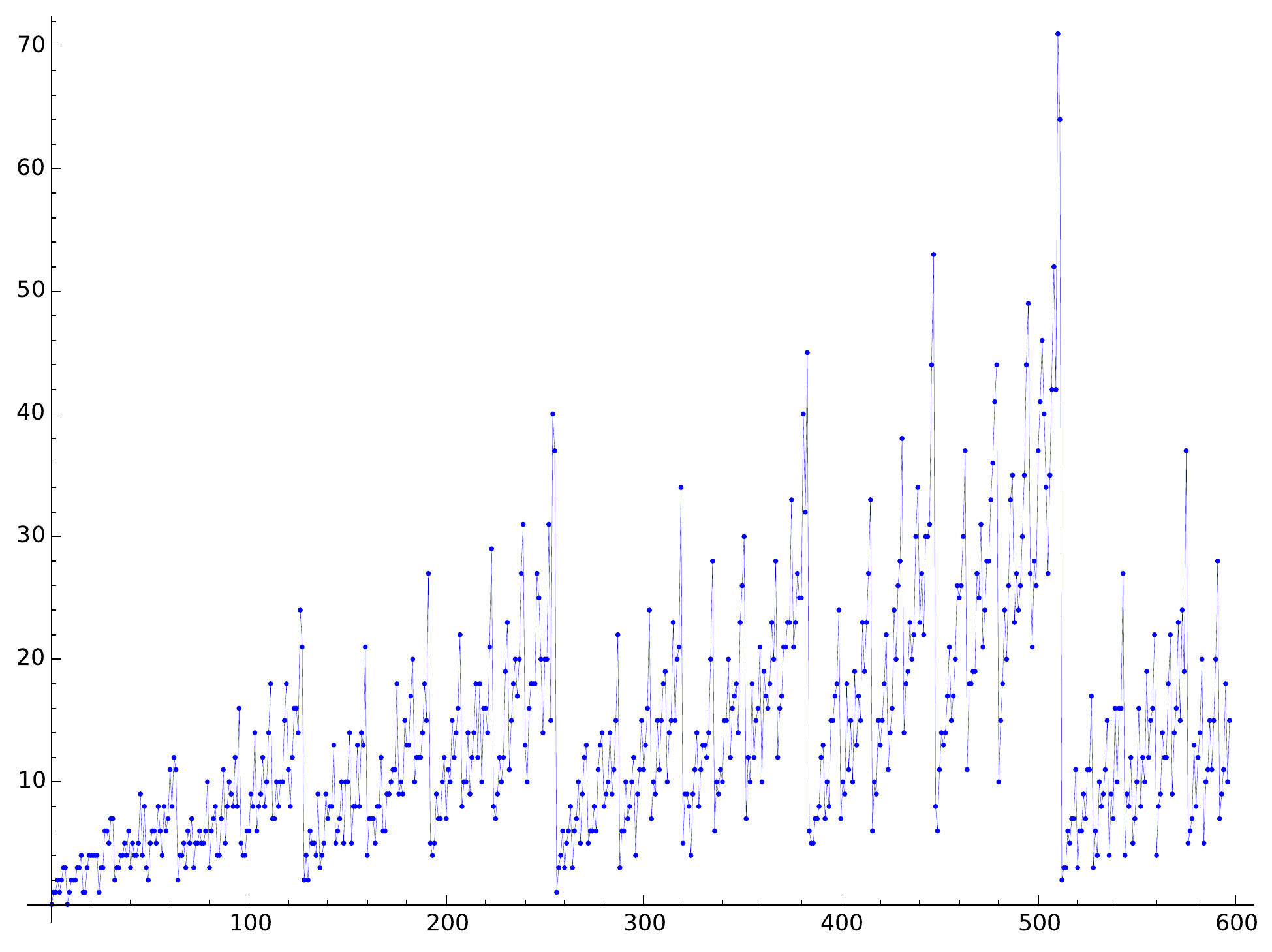}
 \label{FigomegaN}
 }
 }
 \caption{Comparison of the size and structure of the sequence $W_{\NNs}(m), m\ge 1$:\\
     \texttt{Left:} the graph of $\log(W_{\NNs}(m))$; \ \ \texttt{Right:} the graph of
$\omega(W_{\NNs}(m))$. 
}
 \label{FigLoN}
 \end{figure}

More precisely, the closeness between the two sequences can be measured by the surplus number of 
prime factors between $W_{\NNs}(m)$ or $ W_{\PP}(m)$ and their greatest common divisor,
$G(m):=\gcd(W_{\NNs}(m),W_{\PP}(m))$, $m\ge 1$. For this, the appropriate counting functions are
\begin{equation*}
	\begin{split}
	s_\NNs(f;K) &=\# \{1\le m\le K : \omega\big(W_{\NNs}(m)/G(m)\big)=f\}\,,\\
	s_\PP(f;K) &=\# \{1\le m\le K : \omega\big(W_{\PP}(m)/G(m)\big)=f\}\,.
	\end{split}
\end{equation*}
In Table~\ref{TableSurplus} we counted the number of integers $m$ for which the surplus occurs.
Notice that if $m\le 1024$, the largest surplus is $7$. This is small when compared with the 
maximum values of $\omega(W_\NNs(m))$ and $\omega(W_\PP(m))$ in this range, which are equal to 
$\omega(W_\NNs(1023))=130$ and $\omega(W_\PP(1023))=132$.

\begin{table}[!htbp]
\small
\centering
\caption{The surplus counting functions of $W_\NNs$ and $W_\PP$.}\label{TableSurplus}
$
\begin{array}{ccccccccccc}
\toprule
f & 0& 1& 2& 3& 4& 5& 6& 7& 8& 9\\
	\midrule
s_\NNs(f;1024) & 391& 311& 183& 77& 41& 14& 5& 2& 0& 0\\
s_\PP(f;1024) & 353& 391& 186& 74& 11& 6& 3& 0& 0& 0\\
\bottomrule
\end{array}
$
\end{table} 

Equality between $W_{\NNs}(m)$ and $W_{\PP}(m)$ for $m\le 1024$ occurs $149$ times.
We also mention that even in a larger range, integers $m$ for which $W_{\NNs}(m)=W_{\PP}(m)$ tend 
to appear in clusters, often grouping a varying number of consecutive numbers.

\subsection{Primes dividing the maximal values of $W_\NNs(m)$}
An intricate pattern of the sets of primes that divide the larger values of $W_\NNs(m)$ around 
$m=2^g$ develops as $g$  increases. Let us see a typical example, the case $g=8$. To emphasize the 
presence or absence and the position of prime divisors in the list all primes $\le m$, we have 
listed them all, but in two distinguished ways. Thus $W_\NNs(m)$ is the product of primes written 
in normal 
font, while the primes that do not 
divide $W_\NNs(m)$ are shown in red color (in the electronic form) smaller font. Thus, we have:
\begin{equation*}
	\begin{split}
	W_\NNs(255) :
		\ &\rs{2}, 3, \rs{5}, 7, 11, 13, \rs{17}, 19, \rs{23, 29, 31}, 37, 41, 
43, 47, \rs{53, 59, 61, 67, 71, 73, 79, 83},\\ 
	& 89, 97, 101, 103, 107, 109, 113, 127, 131, 137, 139, 149, 151, 157, 163, 167,\\
		&173, 179, 181, 191, 193, 197, 199, 211, 223, 227, 229, 233, 239, 241, 251;\\[1.3ex]
	W_\NNs(256):\ &	
	2, \rs{3}, 5, 7, 11, \rs{13}, 17, 19, 23, \rs{29, 31, 37, 41}, 43, 47, \rs{53, 59, 61}, 67, 71, 
73, 79, 83,\\
	& \rs{89, 97, 101, 103, 107, 109, 113, 127}, 131, 137, 139, 149, 151, 157, 163, 167,\\
		&173, 179, 181, 191, 193, 197, 199, 211, 223, 227, 229, 233, 239, 241, 251\,,
	\end{split}
\end{equation*}
so $\omega(W_\NNs(255))=40$ and $\omega(W_\NNs(256))=37$.

\subsection{A problem of Sloane}\label{SloaneProblem}
\renewcommand{\thefootnote}{$\ddagger$} 
N. J. A. Sloane~\cite[A222313]{OEIS}, \cite[Question 3]{CZ2014} orders increasingly and eliminates 
duplicates from the terms of $W_\NNs$ and obtains sequence\footnote{Considering only the first 
five hundred terms, Sloane missed $W_\NNs(1025)=41$ and $W_\NNs(513)=57$ from the list of terms of 
size smaller than $100$.}
\begin{equation}\label{eqOrderedNNs}
	\begin{split}
UO(W_\NNs) : \ 	&1, 2, 3, 5, 6, 15, 17, 33, 41, 55, 57, 65, 70, 105, 129, \\
	& 257, 273, 385, 561, 897, 969, 1001,\dots
	\end{split}
\end{equation}
He asks if the first part of the list contains all numbers $\le 100$ that appear in this 
sequence. The numbers from \eqref{eqOrderedNNs} are obtained from the first $8200$ terms of 
$W_\NNs$.

Examining the terms, we observed a general formula for a numbers that make a big jump in the 
beginning, during the process of ordering.
\begin{conjecture}\label{Conjecture3}
For any integer $g\ge 0$, we have:
	\begin{equation*}
	\begin{split}
	W_\NNs(2^g+1)=
	\begin{cases}
	     2^g+1, &\text{ if\ \ $2^g+1$ is square free}\\
			 (2^g+1)/D_g, &\text{ else,}\\  
	\end{cases}
	\end{split}
\end{equation*}
where $D_g$ is the largest square that divides $2^g+1$.     
\end{conjecture}
For small ranks, if $g\le 13$, Conjecture~\ref{Conjecture3} verifies, since 
$2^g+1$ is prime, for $g=1, 2, 4, 8$, or a product of two distinct primes, for 
$g= 5, 6, 7, 11, 12, 13$, and $W_\NNs(2^g+1)=2^g+1$ in these cases.
For the remaining values, 
we have:
	$W_\NNs(2^3+1)=W_\NNs(3^2)=1$;
	$W_\NNs(2^8+1)=W_\NNs(3^3 \cdot 19)=57$;	
	$W_\NNs(2^{10}+1)=W_\NNs(5^2 \cdot 41)=41$.

Verifying the decomposition of a few hundred more numbers of the form $2^g+1$ and assuming, 
confer the above discussion, that the smallest local minimums of $W_\NNs$ are attained at these 
ranks, we should expect a positive answer to Sloane's question.

\begin{paradoxproblem}\label{Paradox2}
Explain the peculiarity: why, given that the towers of the $p$-tomographies are 
higher when starting with the initial generation $\NNs$ instead of $\PP$, more non-zero cells 
appear 
on the western edge in the second case. For example, counting only terms less than $1000$, we find 
$27$ terms in $OU(W_\PP)$ and $21$ terms in $OU(W_\NNs)$. 
\end{paradoxproblem}

\section*{Acknowledgement}
All calculations and images presented in this work were made using the free open-source mathematical
software system~\cite{SAGE}.

\linespread{1.04}

\end{document}